\documentclass[12pt,a4paper,dvipdfm]{article}
\usepackage[margin=1 in]{geometry}

\usepackage{graphicx}
\usepackage{amssymb, amsmath}
\usepackage{amsthm}
\usepackage{dsfont}
\usepackage{enumerate}
\usepackage{setspace}
\begin{document}
\newcommand{\up}{\vspace*{-0.05cm}}
\allowdisplaybreaks[1]
\numberwithin{equation}{section}
\renewcommand{\theequation}{\thesection.\arabic{equation}}
\newtheorem{thm}{Theorem}[section]
\newtheorem{lemma}{Lemma}[section]
\newtheorem{pro}{Proposition}[section]
\newtheorem{prob}{Problem}[section]
\newtheorem{quest}{Question}[section]
\newtheorem{ex}{Example}[section]
\newtheorem{cor}{Corollary}[section]
\newtheorem{conj}{Conjecture}[section]
\newtheorem{cl}{Claim}[section]
\newtheorem{df}{Definition}[section]
\newtheorem{rem}{Remark}[section]
\newcommand{\beq}{\begin{equation}}
\newcommand{\eeq}{\end{equation}}
\newcommand{\<}[1]{\left\langle{#1}\right\rangle}
\newcommand{\be}{\beta}
\newcommand{\ee}{\end{enumerate}}
\newcommand{\Bul}{\mbox{$\bullet$ } }
\newcommand{\al}{\alpha}
\newcommand{\ep}{\epsilon}
\newcommand{\si}{\sigma}
\newcommand{\om}{\omega}
\newcommand{\la}{\lambda}
\newcommand{\La}{\Lambda}
\newcommand{\Ga}{\Gamma}
\newcommand{\ga}{\gamma}
\newcommand{\im}{\Rightarrow}
\newcommand{\2}{\vspace{.2cm}}
\newcommand{\es}{\emptyset}

\vspace{-2cm}
\markboth{R. Rajkumar and P. Devi}{Toroidal and projective-planar permutability graphs}
\title{\LARGE\bf Classification of finite groups with toroidal or projective-planar permutability graphs}
\author{R. Rajkumar\footnote{e-mail: {\tt rrajmaths@yahoo.co.in}},\ \ \
P. Devi\footnote{e-mail: {\tt pdevigri@gmail.com}}\\
{\footnotesize Department of Mathematics, The Gandhigram Rural Institute -- Deemed University,}\\ \footnotesize{Gandhigram -- 624 302, Tamil Nadu, India}\\[3mm]
Andrei Gagarin\footnote{e-mail: {\tt andrei.gagarin@rhul.ac.uk}}\\
{\footnotesize Department of Computer Science, Royal Holloway, University of London,}\\
{\footnotesize Egham, Surrey, TW20 0EX, UK}\vspace{1mm}\\
}
\date{}
\maketitle
\doublespacing
\begin{abstract}
Let $G$ be a group. \textit{The permutability graph of subgroups of $G$}, denoted by $\Gamma(G)$, is a graph having all the proper subgroups of $G$ as its vertices, and
two subgroups are adjacent in $\Gamma(G)$ if and only if they permute. In this paper, we classify the finite groups whose permutability graphs are toroidal or projective-planar.
In addition, we classify the finite groups whose permutability graph does not contain one of $K_{3,3}$, $K_{1,5}$, $C_6$, $P_5$, or $P_6$ as a subgraph.

\paragraph{Keywords:}Permutability graph, finite groups, genus, toroidal graph, nonorientable genus, projective-planar graph.
\paragraph{2010 Mathematics Subject Classification:} 05C25,  05C10.
\end{abstract}

\medskip
\section{Introduction}
\label{intro}
\noindent Various algebraic structures are subject of research in algebra. For example, groups, rings, fields, modules, etc. One of the ways to study properties of these algebraic structures is by using tools of graph theory. That is, by suitably defining a graph associated with an algebraic structure, we can study some specific algebraic properties of the structure by analyzing the graph and using graph-theoretic concepts. This has shown to be a fruitful approach in the field of algebraic combinatorics and, in the recent years, has been a topic of interest among researchers. In particular, there are various graphs that have been associated with groups. For instance, see~\cite{abdolla,cameron,will}. Also, several recent papers \cite{hung, MWY2012,W2006} deal with embeddability of graphs, associated with algebraic structures, on topological surfaces.\\
In \cite{asc}, Aschbacher defined a graph corresponding to a group $G$ as follows:
for a fixed prime $p$, all the subgroups of order $p$ in $G$ are vertices of the graph,
and two vertices are adjacent if the corresponding subgroups permute.
In this direction, to study the transitivity of permutability of subgroups of groups, Bianchi et al. \cite{binachi 2} defined a graph corresponding to a group $G$, called the \emph{permutability graph of the non-normal subgroups} of $G$, having all the non-normal subgroups of G as its vertices and two vertices adjacent if and only if the two corresponding subgroups permute. They mainly focused on the number of connected components and the diameter of this graph. Further investigations on this graph can be found in \cite{binachi 1, binachi}.
The authors in \cite{raj} consider a more general setting by associating a given group $G$ with a graph denoted by $\Gamma(G)$. The graph $\Gamma(G)$ is called the \emph{permutability graph of subgroups} of G, having the vertex set consisting of all proper subgroups of G, and two vertices $H$ and $K$ are adjacent in $\Gamma(G)$ if and only if $H$ and $K$ permute in $G$. In \cite{raj}, the authors mainly classify the finite groups whose permutability graph of subgroups is planar.
\begin{thm}\label{genus 19}(\cite[Theorem~5.1]{raj})
 Let $G$ be a finite group. Then $\Gamma(G)$ is planar if and only if $G$ is isomorphic to one of the following groups (where $p$ and $q$ are distinct primes):
 $\mathbb Z_{p^{\alpha}} (\alpha= 2, 3, 4, 5)$, $\mathbb Z_{{p^ \alpha}q} (\alpha =1, 2)$, $\mathbb Z_{p} \times \mathbb Z_{p} (p=2, 3)$, $Q_8$,
 $\mathbb Z_q  \rtimes \mathbb Z_p$, $A_4$, or
 $\mathbb Z_q \rtimes_{2} \mathbb Z_{p^{2}} = \langle a,b~|~a^q= b^{p^2}= 1, bab^{-1}= a^i, {ord_{q}}(i)= p^2 \rangle$ with $p^2~|~(q-1)$.
\end{thm}

A natural question in this direction is to describe the groups having their permutability graph of subgroups of genus one, that is toroidal, or of nonorientable genus one, that is projective-planar. Another motivation for investigating the toroidality and projective-planarity of a permutability graph of subgroups of groups is to exhibit some structure of permutability of subgroups within a given group.
In this paper, we solve these problems in the case of finite groups by classifying the finite groups whose permutability graph of subgroups is toroidal or projective-planar (see Theorem~\ref{t16} in Section~\ref{sec:6} below).
In particular, we show that all the projective-planar permutability graphs are toroidal, which is not the case for arbitrary graphs (e.g., see pp.\,367-368 and Figure~13.33 in \cite{KK2005}).
As a consequence of this research, we also classify finite groups whose permutability graph of subgroups is in some class of graphs characterized by a forbidden subgraph (see Corollary~\ref{genus 20}), which is one of the main applications of these results for group theory. Finally, we formulate related questions for infinite groups.


\section{Preliminaries and notation} \label{sec:2}

In this section, we first recall some concepts, notation, and results in graph theory, which are used later in the subsequent sections. We use standard basic graph theory terminology and notation (e.g., see \cite{arthur}).
Let $G$ be a simple graph with a vertex set $V$ and an edge set $E$. If any two vertices in $G$ are adjacent, then it is called a
\textit{complete graph}. A complete graph on $n$ vertices is denoted by $K_n$.  $G$ is said to be \textit{bipartite} if $V$ can be partitioned into two subsets $V_1$ and $V_2$ such that every edge of $G$ joins a vertex of $V_1$ to a vertex of $V_2$. Then $(V_1, V_2)$ is called a \textit{bipartition} of $G$.
Moreover, if every vertex of $V_1$ is adjacent to every vertex of $V_2$, then $G$ is called \textit{complete bipartite} and denoted by $K_{m,n}$, where $|V_1|=m$, $|V_2|=n$ (without loss of generality, $m\le n$).
A \textit{path} connecting two vertices $u$ and $v$ in $G$ is a finite sequence  $(u=) v_0, v_1, \ldots, v_n (=v)$ of distinct vertices (except,
possibly, $u$ and $v$) such that $u_i$ is adjacent to $u_{i+1}$ for all $i=0, 1, \ldots , n-1$.
A path is a \emph{cycle} if $u=v$. The length of a path or a cycle is the number of edges in it. A path or a cycle of length $n$ is denoted by $P_n$ or $C_n$, respectively. We define a graph $G$ to be \emph{$X$-free} if it does not contain a subgraph isomorphic to a given graph $X$. $\overline{G}$ denotes the complement of a graph $G$, and, for an integer $q\ge 1$, $qG$ denotes the graph composed of $q$ disjoint copies of $G$. For two graphs $G$ and $H$, $G\cup H$ denotes a disjoint union of $G$ and $H$, $G+H$ denotes a graph with the vertex set composed of the vertices of $G$ and $H$ and the edge set composed of the edges of $G$ and $H$ plus all the edges $uv$ such that $u\in G$ and $v\in H$.


A graph is said to be \textit{embeddable} on a topological surface if it can be drawn on the surface in such a way that no two edges cross.
The (orientable) \textit{genus} of a graph $G$, denoted by $\gamma(G)$, is the smallest non-negative integer $n$ such that $G$ can be embedded on the sphere with $n$ handles.
A graph is \textit{planar} if its genus is zero and \textit{toroidal} if its genus is equal to one.
For non-orientable topological surfaces (e.g., the projective plane, Klein bottle, etc.),
the \textit{nonorientable genus} of $G$ is the smallest integer $q$ such that $G$ can be embedded on the sphere with $q$ crosscaps, and it is denoted by $\overline{\gamma}(G)$.
The projective plane is the sphere with one crosscaps, and can be represented by a disk with antipodal (opposite) points on its boundary identified.
Respectively, a graph is \textit{projective-planar} if its nonorientable genus is equal to one.

A \textit{topological obstruction} for a surface is a graph $G$ of minimum vertex degree at least three such that $G$ does not embed on the surface,
but $G-e$ is embeddable on the surface for every edge $e$ of $G$. A \emph{minor-order obstruction} $G$ is a topological obstruction with the additional
property that, for each edge $e$ of $G$, $G$ with the edge $e$ contracted embeds on the surface.
Let $v$ be a vertex of degree three in a graph $G$, adjacent to (distinct) vertices $v_1$, $v_2$, $v_3$. Then a \textit{wye-delta}
transformation of $G$ is the operation of deleting $v$ and adding the edges of triangle with the vertex set $\{v_1$, $v_2$, $v_3\}$ in $G$.
It is known that wye-delta transformations preserve embeddability of graphs in a given topological surface, i.e. they have embedding-hereditary properties (for example, see \cite{arch}). In particular, the class of graphs embeddable on the torus (i.e. toroidal and planar graphs embedded on the torus) is closed under wye-delta transformations.

The following results are used in the forthcoming sections.
\begin{thm} \label{genus 100} (\cite[Theorems 6.37, 6.38 and 11.19, 11.23]{arthur})
\begin{enumerate}[{\normalfont (1)}]
\item $\gamma(K_n)=\displaystyle\left\lceil \frac{(n-3)(n-4)}{12}\right\rceil$, $n\geq 3$;

    $\gamma(K_{m,n})=\displaystyle\left\lceil \frac{(m-2)(n-2)}{4}\right\rceil$, $m$, $n\geq 2$.
\item $\overline{\gamma}(K_n)=\left\{
   	\begin{array}{ll}
   		\displaystyle\left\lceil \frac{(n-3)(n-4)}{6} \right\rceil, & \mbox{~~if~ } n\geq 3, n\neq 7;  \\
   		3, & \mbox{~~if~ } n=7;
   	\end{array}
   \right.$

$\overline{\gamma}(K_{m,n})=\displaystyle\left\lceil \frac{(m-2)(n-2)}{2}\right\rceil$, $m$, $n\geq 2$.
\end{enumerate}
\end{thm}
As a consequence of Theorem~\ref{genus 100}, one can see that $\gamma(K_n)>1$ for $n\geq 8$, $\overline{\gamma}(K_n)>1$ for $n\geq 7$, $\gamma(K_{m,n})>1$  if either $m\geq4$, $n\geq 5$ or $m\geq 3$, $n\geq 7$, and $\overline{\gamma}(K_{m,n})>1$ if either $m\geq3$, $n\geq5$ or $m=n=4$.


Neufeld and Myrvold \cite{pra} have shown the following.
\begin{thm}(\cite{pra})\label{genus 501}
There are exactly three eight-vertex obstructions $\mathcal{A}_1, \mathcal{A}_2, \mathcal{A}_3$ for the torus, each of them being topological and minor-order (see Figure~\ref{fig:8vertex}).
\end{thm}

\begin{figure}[ ht ]
 \begin{center}
 \includegraphics[scale=.7]{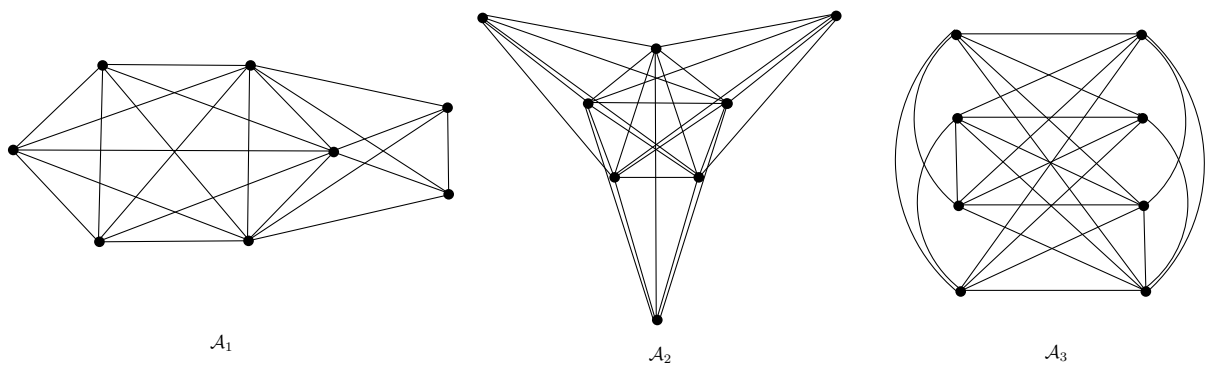}
 \caption{The eight-vertex obstructions for the torus.\label{fig:8vertex}}
 \end{center}
 \end{figure}


Gagarin \emph{et al.}~\cite{and} have found all the toroidal obstructions for the graphs containing no subdivisions of $K_{3,3}$ as a subgraph. These graphs coincide with the graphs containing no $K_{3,3}$-minors and are called \emph{with no $K_{3,3}$'s}.

\begin{thm}(\cite{and}) \label{genus 500}
There are exactly four minor-order obstructions with no $K_{3,3}$'s for the torus, precisely,
$\mathcal{B}_1,\mathcal{B}_2,\mathcal{B}_3,\mathcal{B}_4$ shown in Figure~\ref{fig:noK33}.
\end{thm}

\begin{figure}[ ht ]
 \begin{center}
 \includegraphics[scale=.7]{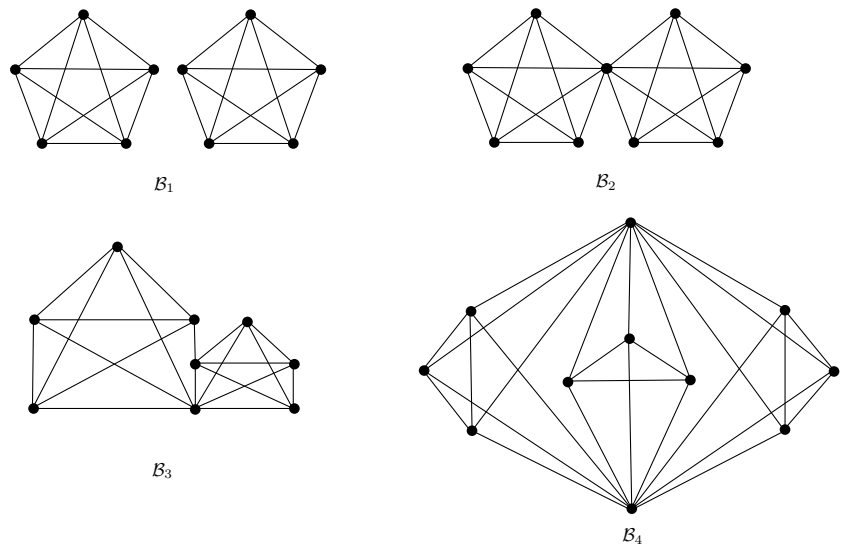}
 \caption{The minor-order obstructions with no $K_{3,3}$'s for the torus.\label{fig:noK33}}
 \end{center}
 \end{figure}

Notice that all the obstructions in Theorems~\ref{genus 501} and \ref{genus 500} are obstructions for toroidal graphs in general, which are very numerous (e.g., see \cite{and}). In other words, for example, the obstructions with no $K_{3,3}$'s of Theorem~\ref{genus 500} can be minors or subgraphs in non-toroidal graphs containing $K_{3,3}$ as a subgraph as well.


Also, we remind here some notions and terminology of group theory.
For any integer $n\geq 3$, the  Dihedral group of order $2n$ is given by $D_{2n}=\langle a, b~|~a^n=b^2=1, ab=ba^{-1}\rangle$.
For any integer $n \geq 2$, the generalized Quaternion group of order $2^n$ is given by
$Q_{2^n} = \big < a, b ~|~a^{2^{n-1}} = b^4 = 1, a^{2^{n-2}} = b^2 = 1,
bab^{-1} = a^{-1}\big >$. For any $\alpha\geq 3$ and a prime $p$, the Modular group of order $p^\alpha$ is given by
$M_{p^\alpha}=\langle a,b~|~a^{p^{\alpha-1}}=b^p=1, bab^{-1}=a^{p^{\alpha-2}+1}\rangle$. $S_n$ and $A_n$ are symmetric and alternating groups of degree $n$, respectively.
We denote by $\text{ord}_n(a)$ the order of an element $a \in \mathbb Z_n$. The number of Sylow $p$-subgroups of a group $G$
is denoted by $n_p(G)$. Recall that $SL_m (n)$ is the group of $m \times m$ matrices having determinant equal to $1$,
whose entries lie in a field with $n$ elements, and that $L_m (n) = SL_m (n)/H$,
where $H = \{kI|k^m = 1\}$. For any prime $q\geq 3$, the  Suzuki group is denoted by $Sz(2^{q})$.


\section{Finite abelian groups}\label{sec:3}
In this section, we classify the finite abelian groups whose permutability graph of subgroups is either toroidal or projective.

Note that the only groups having no proper subgroups are the trivial group and the groups of prime order.
This implies that the graph $\Gamma (G)$ is defined only when the group $G$ is not isomorphic to the trivial group or a group of prime order.

First we recall the following basic result.
\begin{lemma}(\cite[Lemma 3.1]{raj})\label{l1}
 If $G$ is a finite abelian group, then $\Gamma(G) \cong K_r$, where $r$ is the number of proper subgroups of $G$.
\end{lemma}

\begin{pro}\label{genus 2}
Let $G$ be a finite abelian group, and $p,q,r$ are distinct primes. Then
\begin{enumerate}[\normalfont (1)]
\item $\Gamma (G)$ is toroidal if and only if $G$ is isomorphic to one of the following groups:
$\mathbb Z_{p^\alpha} (\alpha=6,7,8)$, $\mathbb Z_{p^3 q}$, $\mathbb Z_{p^2q^2}$, $\mathbb Z_{pqr}$, $\mathbb Z_4\times \mathbb Z_2$, $\mathbb Z_5\times \mathbb Z_5$;
\item $\Gamma (G)$ is projective-planar if and only if $G$ is isomorphic to one of the following groups:
$\mathbb Z_{p^\alpha} (\alpha=6,7)$, $\mathbb Z_{p^3 q}$, $\mathbb Z_{pqr}$, $\mathbb Z_4\times \mathbb Z_2$, $\mathbb Z_5\times \mathbb Z_5$.
\end{enumerate}
\end{pro}
\begin{proof}
 We break the proof into two cases:

\noindent \textbf{Case 1:} Suppose $G$ is cyclic, and $|G|=p_1^{\alpha_ 1}p_2^{\alpha_ 2}\dots p_k^{\alpha_ k}$, where $p_i$'s are distinct primes, $\alpha_i \geq 1$ are integers, $i=1,\ldots,k$. Then the number of distinct subgroups of $G$ is the number of distinct positive divisors of $|G|$.
Thus, by Lemma ~\ref{l1},  we have
\begin{equation}\label{e1}
 \Gamma(G) \cong K_r,
\end{equation}
where $r= (\alpha_{1} + 1)(\alpha_{2} + 1) \cdots (\alpha_{k} +1)- 2$.

This implies $\Gamma(G)$ is toroidal if and only if
$r=5$, 6, 7. This is true when one of the following holds:
\begin{itemize}
\item [\textit{(i)}] $k=1$ with $\alpha_1=6$, $7$, $8$;

\item [\textit{(ii)}] $k=2$ with $\alpha_1=3,\alpha_2=1$;

\item [\textit{(iii)}] $k=2$ with $\alpha_1=\alpha_2= 2$;

\item [\textit{(iv)}]$k=3$ with $\alpha_1=\alpha_2=\alpha_3=1$.
\end{itemize}
Thus, for toroidal $\Gamma(G)$, $G$ is isomorphic to one of $\mathbb Z_{p^\alpha} (\alpha=6$, 7, $8)$, $\mathbb Z_{p^3 q}$, $\mathbb Z_{p^2q^2}$, $\mathbb Z_{pqr}$.

Respectively, $\Gamma(G)$ is projective-planar if and only if
$r=5$, 6. This is true when one of the following holds:
\begin{itemize}
\item [\textit{(i)}] $k=1$ with $\alpha_1=6$, $7$;

\item [\textit{(ii)}] $k=2$ with $\alpha_1=3,\alpha_2=1$;

\item [\textit{(iii)}]$k=3$ with $\alpha_1=\alpha_2=\alpha_3=1$.
\end{itemize}
Thus, for projective-planar $\Gamma(G)$, $G$ is isomorphic to one of $\mathbb Z_{p^\alpha} (\alpha=6$, $7)$, $\mathbb Z_{p^3 q}$, $\mathbb Z_{pqr}$.

\noindent \textbf{Case 2:} Suppose $G$ is non-cyclic. Then we split this case into the following subcases:

\noindent \textbf{Subcase 2a:} $G\cong \mathbb Z_p\times \mathbb Z_p$. Then the number of proper subgroups of $\mathbb Z_p\times \mathbb Z_p$ is $p+1$;
they are $\langle (1,0) \rangle$, and $\{\langle (x,1)\rangle ~|~ x \in \{0,1,\ldots ,p-1\} \}$. By Lemma~\ref{l1}, we have
\begin{equation}\label{e333}
\Gamma(G)\cong K_{p+1}.
\end{equation}
 It follows that $\Gamma(G)$ is toroidal or projective-planar only when $p= 5$.

\noindent \textbf{Subcase 2b:} $G\cong \mathbb Z_{p^2}\times \mathbb Z_p$. If $p=2$, then $\langle (1,0)\rangle$, $\langle (1,1)\rangle$, $\langle (2,0)\rangle$,
$\langle (0,1)\rangle$, $\langle (2,1)\rangle$ and $\langle (2,0),(0,1)\rangle$ are the only proper subgroups of $G$. Therefore,
\begin{equation}\label{e6}
 \Gamma(G)\cong K_6
\end{equation}
and so $\Gamma(G)$ is toroidal and projective-planar. If $p\geq3$, then the proper subgroups $\langle (1,0)\rangle$, $\langle (1,1)\rangle$, $\langle (1,p-1)\rangle$, $\langle (p,0)\rangle$,
$\langle (p,1)\rangle$, $\langle (p,p-1)\rangle$, $\langle (p,0),(0,1)\rangle$ and $\langle (0,1)\rangle$ of $G$ form
$K_8$ as a subgraph of $\Gamma(G)$, implying $\gamma(\Gamma(G))>1$, $\overline{\gamma}(\Gamma(G))>1$.

\noindent \textbf{Subcase 2c:} $G\cong \mathbb Z_{p^2}\times \mathbb Z_{p^2}$. If $p=2$, then $G$ has two subgroups $H:=\mathbb Z_{p^2}\times \mathbb Z_p$, $N:=\mathbb Z_{p^2}$.
Then, by Subcase 2b, $H$ together with its subgroups and $N$ form $K_8$ as a subgraph of $\Gamma(G)$. Therefore, $\gamma(\Gamma(G))>1$, $\overline{\gamma}(\Gamma(G))>1$.

If $p>2$, then we consider the subgroup
$\mathbb Z_{p^2}\times \mathbb Z_p$ of $G$. By Subcase 2b, in this case, $\gamma(\Gamma(\mathbb Z_{p^2}\times \mathbb Z_p))>1$, $\overline{\gamma}(\Gamma(\mathbb Z_{p^2}\times \mathbb Z_p))>1$, and so $\gamma(\Gamma(G))>1$, $\overline{\gamma}(\Gamma(G))>1$.

\noindent \textbf{Subcase 2d:} $G\cong \mathbb Z_{p^k}\times \mathbb Z_{p^l}$, $k$, $l\geq2$. Then $\mathbb Z_{p^2}\times \mathbb Z_{p^2}$ is a
subgroup of $G$. By Subcase 2c, $\gamma(\Gamma(\mathbb Z_{p^2}\times \mathbb Z_{p^2}))>1$, $\overline{\gamma}(\Gamma(\mathbb Z_{p^2}\times \mathbb Z_{p^2}))>1$ implying
$\gamma(\Gamma(G))>1$, $\overline{\gamma}(\Gamma(G))>1$.

\noindent \textbf{Subcase 2e:} $G\cong \mathbb Z_p\times \mathbb Z_{pq}$. Then
$H:=\mathbb Z_p\times \mathbb Z_p$, $H_1:=\mathbb Z_q\times \{e\}$ are subgroups of $G$,
and $H$ has at least three subgroups of order $p$, say $H_2$, $H_3$, $H_4$. Now $H_5:=H_2H_1$, $H_6:=H_3H_1$, and $H_7:=H_4H_1$ are subgroups of $G$.
So, $H$ and $H_i$, $i=1$, 2, $\ldots,8$, form $K_8$ as a subgraph of $\Gamma(G)$, and so $\gamma(\Gamma(G))>1$, $\overline{\gamma}(\Gamma(G))>1$.


\noindent \textbf{Subcase 2f:} $G\cong \mathbb Z_p\times \mathbb Z_p\times \mathbb Z_p$. Then the proper subgroups
$\langle (1,0,0)\rangle$, $\langle (0,1,0)\rangle$, $\langle (0,0,1)\rangle$,
$\langle (1,1,0)\rangle$, $\langle (1,0,1)\rangle$, $\langle (0,1,1)\rangle$, $\langle (0,1,0),(0,0,1)\rangle$ and $\langle (1,0,0),(0,0,1)\rangle$
of $G$ form $K_8$ as a subgraph of $\Gamma(G)$, and so $\gamma(\Gamma(G))>1$, $\overline{\gamma}(\Gamma(G))>1$.

\noindent \textbf{Subcase 2g:} $G\cong \mathbb Z_{p_1^{\alpha_1}}\times \mathbb Z_{p_2^{\alpha_2}}\times \ldots \times\mathbb Z_{p_k^{\alpha_k}}$, where $k\geq 3$, $p_i$'s are primes and at least two of them are equal (since $G$ is non-cyclic, all the primes cannot be distinct here), $\alpha_i\geq 1$ are integers, $i=1,\ldots,k$.
Then, for some $i$ and $j$, $i\neq j$, $G$ has one of  $\mathbb Z_{p_i}\times \mathbb Z_{p_ip_j}$ or $\mathbb Z_{p_i}\times \mathbb Z_{p_i}\times \mathbb Z_{p_i}$ as a
subgroup.
By Subcases 2e, 2f above, the permutability graphs of subgroups of these groups are not toroidal or projective-planar. Thus, it follows that $\gamma(\Gamma(G))>1$, $\overline{\gamma}(\Gamma(G))>1$.

Combining all the above cases together completes the proof.
\end{proof}


\section{Finite non-abelian groups}\label{sec:4}
In this section, we classify the finite non-abelian groups whose permutability graph of subgroups is toroidal or projective-planar.
We first consider the solvable  groups, and then we investigate the non-solvable groups.

\subsection{Solvable groups}
\begin{pro}\label{genus 5}
 Let $G$ be a non-abelian group of order $p^\alpha$, where $p$ is a prime and $\alpha\geq3$. Then $\gamma(\Gamma(G))> 1$, $\overline{\gamma}(\Gamma(G))>1$.
\end{pro}
\begin{proof}
We divide the proof into two cases.

\noindent\textbf{Case 1:} $\alpha=3$. If $p=2$, then the only non-abelian groups of order 8 are $Q_8$ and $M_8$. By Theorem~\ref{genus 19}, $\Gamma(Q_8)$
is planar, and it is shown in \cite[Theorem~4.3]{raj} that
\begin{equation}\label{e111}
 \Gamma(Q_8)=K_4.
\end{equation}
If $G\cong M_8$, then $\langle a\rangle$, $\langle a^2\rangle$, $\langle b\rangle$, $\langle ab\rangle$,
$\langle a^2b\rangle$, $\langle a^3b\rangle$, $\langle a^2,b\rangle$, $\langle a^2,ab\rangle$ are the only proper subgroups of $G$. Here
$\langle a\rangle$, $\langle a^2,b\rangle$, $\langle a^2,ab\rangle$, $\langle a^2\rangle$ are normal in $G$;
$\langle b\rangle$ permutes with $\langle a^2b\rangle$;
$\langle ab\rangle$ permutes with $\langle a^3b\rangle$; no two remaining subgroups permutes. Therefore,
\begin{equation}\label{e3}
\Gamma(G)\cong K_4+\overline{K}_{2,2}.
\end{equation}
Thus, $\Gamma(G)$ has a subgraph $\mathcal A_1$ shown in Figure~\ref{fig:8vertex} (edges $\langle a^2\rangle \langle b \rangle$ and $\langle a^2 \rangle\langle a^2b \rangle$ are removed, see Figure~\ref{fig:subgroups}),
 which is a topological obstruction for the torus, and $\gamma(\Gamma(G))> 1$. Moreover, $\Gamma(G)$ contains $K_{3,5}$ as a subgraph, so  $\overline{\gamma}(\Gamma(G))>1$.

 \begin{figure}[ ht ]
 \begin{center}
  \includegraphics[scale=1]{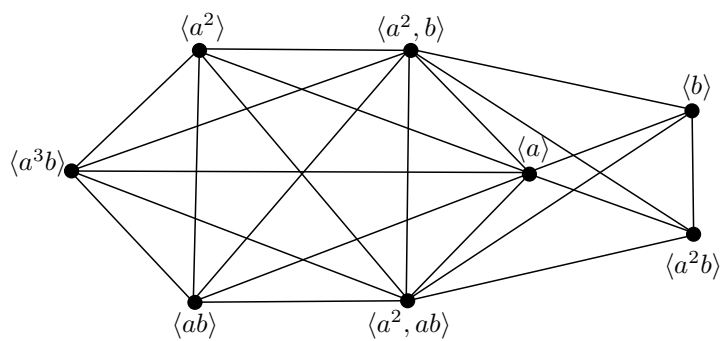}
 \caption{A topological obstruction for the torus.\label{fig:subgroups}}
 \end{center}
 \end{figure}


If $p\neq2$, then we have, up to isomorphism, only two groups, namely $M_{p^3}$ and $(\mathbb Z_p\times \mathbb Z_p)\rtimes \mathbb Z_p$ of order $p^3$.
If $G\cong M_{p^3}$, then $\langle a\rangle$, $\langle ab\rangle$, $\langle ab^{p-1}\rangle$, $\langle a^p,b\rangle$,
$\langle a^p\rangle$, $\langle a^pb^{p-1}\rangle$, $\langle a^p\rangle$, $\langle b\rangle$ are proper subgroups of $G$. Since every pair of subgroups of a
modular group permutes, $K_8$ is a subgraph of $\Gamma(G)$, and so $\gamma(\Gamma(G))>1$, $\overline{\gamma}(\Gamma(G))>1$. If
$G\cong(\mathbb Z_p\times \mathbb Z_p)\rtimes \mathbb Z_p=\langle a,b,c~|~a^p=b^p=c^p=1, ab=ba, ca=ac, cbc^{-1}=ab\rangle$,
then $\mathbb Z_p\times \mathbb Z_p$ is a subgroup of $G$ and, from the proof of
Proposition~\ref{genus 2}, $\gamma(\Gamma(\mathbb Z_p\times \mathbb Z_p))>1$, $\overline{\gamma}(\Gamma(\mathbb Z_p\times \mathbb Z_p))>1$ when $p>5$.
If $p=3$ or 5, then $\langle a\rangle$, $\langle b\rangle$,
$\langle c\rangle$, $\langle a,b\rangle$, $\langle a,c\rangle$, $\langle b,c\rangle$, $\langle ab\rangle$, $\langle ab^2\rangle$, $\langle ac\rangle$,
$\langle a^2c\rangle$ are proper subgroups of $G$. Here $\langle a,b\rangle$, $\langle a,c\rangle$, $\langle b,c\rangle$ are normal in $G$.
It follows that $K_{3,7}$ is a subgraph of $\Gamma(G)$ with bipartition $X:=\{\langle a,b\rangle$,
$\langle a,c\rangle$, $\langle b,c\rangle\}$ and $Y=\{\langle a\rangle$, $\langle b\rangle$,
$\langle c\rangle$, $\langle ab\rangle$, $\langle ab^2\rangle$, $\langle ac\rangle$, $\langle a^2c\rangle\}$, and so $\gamma(\Gamma(G))>1$, $\overline{\gamma}(\Gamma(G))>1$.

\noindent\textbf{Case 2:} $\alpha\geq 4$. By \cite[Theorem IV, p.129]{burn}, $G$ has at least three subgroups, say $H_1$, $H_2$, $H_3$, of order $p^{\alpha-1}$ and at least
three subgroups, say $H_4$, $H_5$, $H_6$, of order $p^{\alpha-2}$.
If $G$ has more than one subgroup of order $p$,
then $G\ncong Q_{2^\alpha}$ by \cite[Proposition 1.3]{scott}, and, by \cite[Theorem IV, p.129]{burn},
we have at least three subgroups of order $p$, say $H_7$, $H_8$, $H_9$.
By \cite[Corollary of Theorem IV, p.129]{burn}, for each divisor of $|G|$, $G$ has at least one normal subgroup of that order.
So, without loss of generality, we assume $H_4$, $H_7$ are normal in $G$.
Since $H_1$, $H_2$, $H_3$ are also normal in $G$, $K_{5,4}$ is a subgraph of $\Gamma(G)$ with bipartition $X:=\{H_1$, $H_2$, $H_3$, $H_4$, $H_7\}$ and $Y:=\{H_5$, $H_6$, $H_8$, $H_9\}$.
Therefore, $\gamma(\Gamma(G))>1$, and, since $\Gamma(G)$ contains $K_{3,5}$, $\overline{\gamma}(\Gamma(G))>1$.

If $G$ has a unique subgroup of order $p$, then
$G\cong Q_{2^\alpha}$ by \cite[Proposition 1.3]{scott}, and so $\langle a\rangle$, $\langle a^2\rangle$, $\langle a^4\rangle$, $\langle b\rangle$, $\langle a^2,b\rangle$,
$\langle a^2,ab\rangle$, $\langle ab\rangle$, $\langle a^2b\rangle$, $\langle a^3b\rangle$ are proper subgroups of $G$. Since $\langle a\rangle$,
$\langle a^4\rangle$, $\langle a^2,b\rangle$, $\langle a^2,ab\rangle$ are normal in $G$, $K_{4,5}$ is a subgraph of $\Gamma(G)$
with bipartition $X:=\{\langle a\rangle$, $\langle a^4\rangle$, $\langle a^2,b\rangle$, $\langle a^2,ab\rangle\}$ and
$Y:=\{\langle a^2\rangle$, $\langle b\rangle$, $\langle ab\rangle$, $\langle a^2b\rangle$, $\langle a^3b\rangle\}$, and so $\gamma(\Gamma(G))>1$, $\overline{\gamma}(\Gamma(G))>1$.
\end{proof}

If $G$ is a non-abelian group of order $pq$, then, by Theorem~\ref{genus 19}, $\Gamma(G)$ is planar, and it is shown in \cite[Theorem~4.4]{raj} that
\begin{equation}\label{222}
\Gamma(G)\cong K_{1,q}.
\end{equation}


Consider the semi-direct product $\mathbb Z_q \rtimes_{t} \mathbb Z_{p^{\alpha}} = \langle a,b~|~a^q= b^{p^{\alpha}}= 1, bab^{-1}= a^i,
{ord_{q}}(i)= p^t \rangle$, where $p$ and $q$ are distinct primes with $p^t~|~(q-1)$, $t \geq 0$. Then every semi-direct product $Z_q \rtimes Z_{p^{\alpha}}$
 is  one of these types \cite[Lemma 2.12]{boh-reid}. Note that, in what follows, we omit the subscript when $t = 1$.

\begin{pro}\label{genus 7}
 Let $G$ be a non-abelian group of order ${p^2}q$, where $p$ and $q$ are distinct primes. Then
 \begin{enumerate}[\normalfont (1)]
 \item $\Gamma(G)$ is toroidal if and only if
$G$ is isomorphic to one of $\mathbb Z_3\rtimes \mathbb Z_4$, $\mathbb Z_5\rtimes \mathbb Z_4$ or $\langle a, b, c~|~ a^p=b^p=c^q=1, ab=ba, cac^{-1}=b^{-1},
cbc^{-1}= ab^{l} \rangle$, where $\bigl(\begin{smallmatrix}
  0 & -1\\ 1 & l
\end{smallmatrix} \bigr)$ has order $q$ in $GL_2(p)$, $p=3,5$;
\item $\Gamma(G)$ is projective-planar if and only if
$G$ is isomorphic to either $\mathbb Z_3\rtimes \mathbb Z_4$ or $\langle a, b, c~|~ a^3=b^3=c^q=1, ab=ba, cac^{-1}=b^{-1},
cbc^{-1}= ab^{l} \rangle$, where $\bigl(\begin{smallmatrix}
  0 & -1\\ 1 & l
\end{smallmatrix} \bigr)$ has order $q$ in $GL_2(3)$.
\end{enumerate}
\end{pro}
\begin{proof}
Here we use the classification of groups of order $p^2q$ given in \cite[pp.\,76-80]{burn}.
We have the following cases to consider:

\noindent\textbf{Case 1:} $ p< q$:

\noindent\textbf{Case 1a:}  $p \nmid (q-1)$. By Sylow's Theorem, it is easy to see that there is no non-abelian groups in this case.

\noindent\textbf{Case 1b:}  $p ~|~ (q-1)$, but $p^2 \nmid (q-1)$. In this case, there are two non-abelian groups.

The first group is $G_1:=\mathbb Z_q\rtimes\mathbb Z_{p^2}=\langle a,b~|~a^q=b^{p^2}=1,{bab}^{-1}=a^i,ord_q(i)=p\rangle$.
It is shown in the proof of Proposition~3.4 in \cite{raj1} that
\begin{equation}\label{e8}
\Gamma(G_1)\cong K_3+\overline{K}_q.
\end{equation}
If $q\geq7$, then $K_{3,7}$ is a subgraph of $\Gamma(G_1)$, and so $\gamma(\Gamma(G_1))>1$, $\overline{\gamma}(\Gamma(G_1))>1$.
Note that $q=5$ is not possible here. If $q=3$, then, by Theorem~\ref{genus 19},
$\Gamma(G_1)$ is non-planar, and it is a subgraph of $K_6$. Therefore, $\Gamma(G_1)$ is toroidal and projective-planar.

The second group in this case is $G_2:=\langle a,b,c~|~a^q=b^p=c^p=1,bab^{-1}=a^i,ca= ac,cb= bc,{ord_q}(i)=p\rangle$.
 Here $H_1:=\langle a\rangle$, $H_2:=\langle b\rangle$, $H_3:=\langle c\rangle$, $H_4:=\langle a,b\rangle$, $H_5:=\langle a,c\rangle$,
$H_6:=\langle b,c\rangle$, $H_7:=\langle ac\rangle$, $H_8:=\langle a^2c\rangle$, $H_9:=\langle bc\rangle$, $H_{10}:=\langle ab\rangle$
are proper subgroups of $G$.
Also, $H_1$, $H_4$, $H_5$ are normal in $G_2$. It
follows that $K_{3,7}$ is a subgraph of $\Gamma(G_2)$ with bipartition $X:=\{H_1$, $H_4$, $H_5\}$ and $Y:=\{H_2$, $H_3$, $H_6$, $H_7$, $H_8$, $H_9$, $H_{10}\}$.
Therefore, $\gamma(\Gamma(G_2))>1$ and $\overline{\gamma}(\Gamma(G_2))>1$.

\noindent\textbf{Case 1c:} $p^2 ~|~ (q-1)$.  In this case, we have both groups $G_1$ and  $G_2$ from Case 1b
 together with the group $G_3:=\mathbb Z_q\rtimes_{2}\mathbb Z_{p^2}=\langle a,b|a^q=b^{p^2}=1,bab^{-1}=a^i,{ord_{q}}(i)=p^2\rangle$.

For the group $G_1$, it is also possible to have $q=5$ and $p=2$ here. Then $G_1 =\mathbb Z_5\rtimes \mathbb Z_4$, and $\gamma(\Gamma(G_1))=1$: a toroidal embedding of $\Gamma(\mathbb Z_5\rtimes \mathbb Z_4)$ is shown in Figure~\ref{fig:embed}. However, since $\Gamma(\mathbb Z_5\rtimes \mathbb Z_4)$ contains $K_{3,5}$ as a subgraph, $\overline{\gamma}(\Gamma(G_1))>1$.

\begin{figure}[ ht ]
\begin{center}
\includegraphics[scale=1]{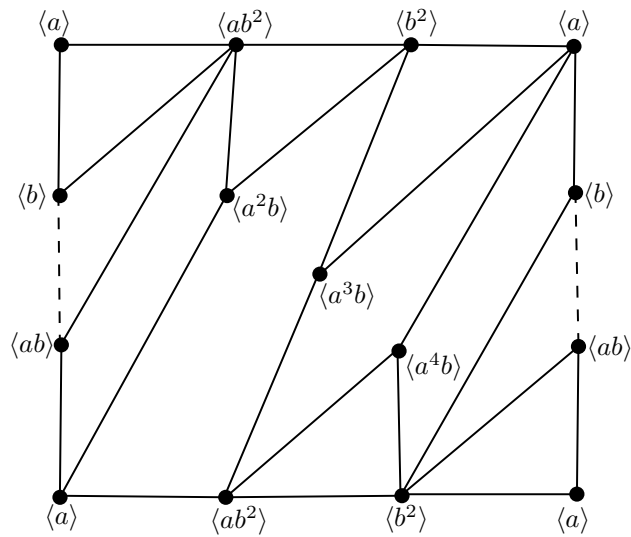}
\caption{A toroidal embedding of $\Gamma (\mathbb Z_5 \rtimes \mathbb Z_4)$.\label{fig:embed}}
\end{center}
\end{figure}
By Theorem~\ref{genus 19}, $\Gamma(G_3)$ is planar, and it is shown in \cite[p.\,7]{raj} that
\begin{equation}\label{e112}
 \Gamma(G_3)=K_2+qK_2.
\end{equation}

\noindent\textbf{Case 2:} $ p> q$:

\noindent\textbf{Case 2a:} $ q ~\nmid~ (p^{2}-1)$. In this case there is no non-abelian groups.

\noindent\textbf{Case 2b:} $ q ~|~ (p-1)$. We have two groups in this case. The first is
$G_4 :=\mathbb Z_{p^2}\rtimes \mathbb Z_q=\langle a,b~|~a^{p^2}= b^q=1,bab^{-1}=a^i,{ord_{p^2}}(i)=q\rangle$.
It is shown in the proof of Proposition~3.4 in \cite{raj1} that
\begin{equation}\label{e100}
 \Gamma(G_4)\cong K_2+pK_{1,p}.
\end{equation}
Clearly, $K_2+3K_{1,3}$ is a subgraph of $\Gamma(G_4)$ ($p\ge 3$ here). We show that $\gamma(K_2+3K_{1,3})>1$, implying $\gamma(\Gamma(G_4))>1$.
Consider the graph shown in Figure~\ref{fig:wow}, which is a subgraph of $K_2+3K_{1,3}$.
Since wye-delta transformations preserve embeddability of graphs in the torus (e.g., see \cite{arch}), the class of toroidal (and, respectively, planar) graphs is closed under wye-delta transformations.
However, by using wye-delta transformations, we can transform the graph in Figure~\ref{fig:wow} to the graph $\mathcal B_4$ of Figure~\ref{fig:noK33}, which is an obstruction for the torus.
It follows that $\gamma(K_2+3K_{1,3})>1$.
\begin{figure}[ ht ]
\begin{center}
\includegraphics[scale=.75]{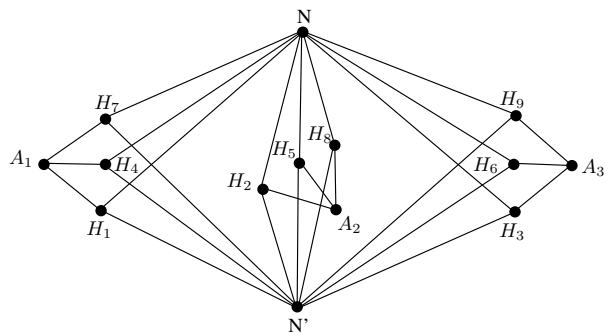}
\caption{A subgraph of $K_2+3K_{1,3}$.\label{fig:wow}}
\end{center}
 \end{figure}
Here $\Gamma(G_4)$ also contains a subgraph shown in Figure~\ref{fig:proj}, which is one of the obstructions for the projective plane (e.g., see Theorem~$0.1$ and graph $D_1$ of case (3.30) on p.\,345 in \cite{glov}). Therefore, $\overline{\gamma}(\Gamma(G_4))>1$.

\begin{figure}[ ht ]
\begin{center}
\includegraphics[scale=1]{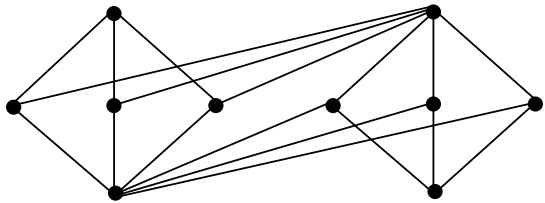}
\caption{An obstruction for the projective plane.\label{fig:proj}}
\end{center}
\end{figure}

Next, we have the family of groups $\langle a,b,c~|~a^p=b^p=c^q=1,cac^{-1}=a^i,cbc^{-1}=b^{i^t},ab= ba,{ord_{p}}(i)=q\rangle$.
 There are $(q+3)/2$ isomorphism types in this family: one for $t=0$, and one for each pair $\{ x, x^{-1}\}$ in $\mathbb F^{\times}_p$.
We will refer to all of these groups of order $p^2q$ as $G_{5(t)}$. Since $\mathbb Z_p\times \mathbb Z_p$ is a subgroup of $G_{5(t)}$, when $p>5$, $\Gamma(G_{5(t)})$ is not toroidal or projective-planar because
$\gamma(\Gamma(\mathbb Z_p\times \mathbb Z_p))>1$, $\overline{\gamma}(\Gamma(\mathbb Z_p\times \mathbb Z_p))>1$ by Proposition~\ref{genus 2}. If $p\leq 5$, then $G_{5(t)}$ has
 $H_1:=\langle a\rangle,H_2:=\langle b\rangle$,
 $H_3:=\langle c\rangle,H_4:=\langle a,b\rangle$, $H_5:=\langle a,c\rangle,H_6:=\langle b,c\rangle$, $H_7:=\langle ab\rangle,H_8:=\langle a^2b\rangle$
as its proper subgroups. Here $H_1$, $H_2$, $H_4$, $H_7$, $H_8$ permute with each other; $H_1$ is a subgroup of $H_5$; $H_2H_5=H_5H_4=
H_5H_7=H_5H_8=H_3H_4=H_6H_1=H_6H_4=H_6H_5=G_{5(t)}$; $H_3H_1=H_5$; $H_3$ is a subgroup of $H_5$. It follows that $\Gamma(G_{5(t)})$ has a subgraph isomorphic to  $\mathcal A_1$
of Figure~\ref{fig:8vertex}, which is an obstruction for the torus, and so $\gamma(\Gamma(G_{5(t)}))>1$.
Also, $\Gamma(G_{5(t)})$ contains a subgraph isomorphic to $K_{3,5}$, implying $\overline{\gamma}(\Gamma(G_{5(t)}))>1$.

\noindent\textbf{Case 2c:}  $q ~|~ (p+1)$. In this case, we have only one group of order
${p^2}q$, given by $G_6:=(\mathbb Z_p\times\mathbb Z_p)\rtimes\mathbb Z_q=\langle a,b,c~|~a^p=b^p=c^q=1,ab=ba,cac^{-1}=a^{i}b^{j},
cbc^{-1}= a^{k}b^{l} \rangle$, where $\bigl(\begin{smallmatrix}
  i & j\\ k & l
\end{smallmatrix} \bigr)$ has order $q$ in $GL_2(p)$.  $G_6$ has a subgroup isomorphic to $\mathbb Z_p \times \mathbb Z_p$, and, when $p>5$, Proposition~\ref{genus 2} gives
$\gamma(\Gamma(\mathbb Z_p\times \mathbb Z_p))>1$, $\overline{\gamma}(\Gamma(\mathbb Z_p\times \mathbb Z_p))>1$, implying $\gamma(\Gamma(G_6))>1$, $\overline{\gamma}(\Gamma(G_6))>1$.

Therefore, we only need to investigate the cases $p= 3$ and $p=5$. First, suppose $G_6$ has a subgroup of order $pq$. If $p=5$, then $H_1:=\langle a\rangle$, $H_2:=\langle b\rangle$, $H_3:=\langle c\rangle$,
 $H_4:=\langle a,b\rangle$, $H_5:=\langle a,c\rangle$, $H_6:=\langle b,c\rangle$, $H_7:=\langle ab\rangle$, $H_8:=\langle a^2b\rangle$,
$H_9:=\langle a^3b\rangle$,
 $H_{10}:=\langle a^4b\rangle$ are proper subgroups of $G_6$. Here $H_4$ is normal in $G_6$; $H_1H_3=H_5;
H_2H_3=H_6;H_1$, $H_2$ permute with $H_7$, $H_8$, $H_9$, $H_{10}$. It follows that $K_{3,7}$ is a subgraph of $\Gamma(G_6)$ with
bipartition $X:=\{H_1$, $H_2$, $H_4\}$ and $Y:=\{H_3$, $H_5$, $H_6$, $H_7$, $H_8$, $H_9$, $H_{10}\}$.
If $p=3$, then $H_1:=\langle a\rangle$, $H_2:=\langle b\rangle$, $H_3:=\langle c\rangle$, $H_4:=\langle a,b\rangle$, $H_5:=\langle a,c\rangle$,
$H_6:=\langle b,c\rangle$, $H_7:=\langle ac\rangle$, $H_8:=\langle a^2c\rangle$, $H_9:=\langle ab\rangle$, $H_{10}:=\langle a^2b\rangle$ are proper
subgroups of $G_6$.
Here $H_1$, $H_2$, $H_4$, $H_9$, $H_{10}$ permute with each other; $H_1H_3=H_5=H_1H_7=H_1H_8$; $H_2H_5=G_6$; $H_2$ is a subgroup of $H_6$;
$H_2H_7=\langle b,ac\rangle$; $H_2H_8=\langle b,a^2c\rangle$; $H_4$ is a normal subgroup of $G_6$. It follows that $\Gamma(G_6)$
contains $K_{3,7}$ as a subgraph with bipartition $X:=\{H_1$, $H_2$, $H_4\}$ and $Y:=\{H_3$, $H_5$, $H_6$, $H_7$, $H_8$, $H_9$, $H_{10}\}$.
Therefore, $\gamma(\Gamma(G_6))>1$, $\overline{\gamma}(\Gamma(G_6))>1$ when $p=3$ or $5$.

If $G_6$ has no subgroup of order $pq$, then  $G_6:=\langle a, b, c~|~ a^p=b^p=c^q=1, ab=ba, cac^{-1}=b^{-1},
cbc^{-1}= ab^{l} \rangle$, where $\bigl(\begin{smallmatrix}
  0 & -1\\ 1 & l
\end{smallmatrix} \bigr)$ has order $q$ in $GL_2(p)$. In this case, $G_6$ has a unique subgroup of order $p^2$, $p+1$ subgroups of order $p$, $p^2$ subgroups of order $q$, and these are the only proper subgroups of $G_6$. It follows that
\begin{equation}\label{e2222}
\Gamma(G_6)\cong K_1+(K_{p+1}\cup \overline{K}_{p^2}),
\end{equation}
where $p=3$, 5. Thus, $\Gamma(G_6)$ is toroidal when $p=3, 5$; also, $\overline{\gamma}(\Gamma(G_6))=1$ when $p=3$, and $\overline{\gamma}(\Gamma(G_6))>1$ when $p=5$.

Note that if $(p,q)= (2, 3)$, the Cases 1 and 2 are not mutually exclusive. Up to isomorphism, there
are three non-abelian groups of order 12: $\mathbb Z_3 \rtimes \mathbb Z_4$, $D_{12}$, and $A_4$.
Here the permutability graphs of subgroups of $\mathbb Z_3 \rtimes \mathbb Z_4$ (the group $G_1$), and  $D_{12}$ (the group $G_2$) are
already dealt with in Case 1b. However, for the case of $A_4 $ (the group $G_6$), by Theorem~\ref{genus 19},
$\Gamma(A_4)$ is planar, and it is shown in \cite[p.\,8]{raj} that
\begin{equation}\label{e4}
\Gamma(A_4)\cong K_1+ (K_3\cup \overline{K}_4).
\end{equation}
Putting all these cases together, the result follows.
\end{proof}

\begin{pro}\label{genus 8}
 If $G$ is a non-abelian group of order $p^\alpha q$, where $p$, $q$ are two distinct primes and $\alpha\geq3$, then $\gamma(\Gamma(G))>1$, $\overline{\gamma}(\Gamma(G))>1$.
\end{pro}
\begin{proof}
Let $P$ denote a Sylow $p$-subgroup of $G$. We shall prove this result by induction on $\alpha$. First we prove this when $\alpha=3$.
If $p>q$, then $n_p=1$, by Sylow's theorem and our group $G\cong P\rtimes \mathbb Z_q$.
If $\gamma(\Gamma(P))>1$, $\overline{\gamma}(\Gamma(P))>1$, then
$\gamma(\Gamma(G))>1$, $\overline{\gamma}(\Gamma(G))>1$, respectively.
Therefore, it is enough to consider the cases when $\gamma(\Gamma(P))\leq 1$, $\overline{\gamma}(\Gamma(G))\leq 1$. By Propositions~\ref{genus 2} and ~\ref{genus 5},
we must have $P\cong \mathbb Z_{p^3}$. Then
$G\cong \mathbb Z_{p^3}\rtimes \mathbb Z_q=\langle a, b~|~a^{p^3}=b^q=1, bab^{-1}=a^i, \mbox{ord}_q(i)=p\rangle$ and
$H_1:=\langle a\rangle$, $H_2:=\langle a^p\rangle$, $H_3:=\langle a^{p^2}\rangle$,
$H_4:=\langle a^p,b\rangle$, $H_5:=\langle a^{p^2},b\rangle$,
 $H_6:=\langle b\rangle$, $H_7:=\langle ab\rangle$, $H_8:=\langle a^2b\rangle$, $H_9:=\langle a^3b\rangle$, $H_{10}:=\langle a^4b\rangle$ are proper
subgroups of $G$. Also, $H_1$, $H_2$, $H_3$ are normal in $G$. It follows that $K_{3,7}$
is a subgraph of $\Gamma(G)$ with bipartition $X:=\{H_1$, $H_2$, $H_3\}$ and $Y:=\{H_4$, $H_5$, $H_6$, $H_7$, $H_8$, $H_9$, $H_{10}\}$.
Therefore, $\gamma(\Gamma(G))>1$, $\overline{\gamma}(\Gamma(G))>1$.

Now, let us consider the case $p< q$ and $(p,q)\neq (2$, $3)$. Here $n_{q} = p$ is not possible. If $n_{q}= p^{2}$, then $q| (p+1)(p-1)$ which implies that
$q| (p+1)$ or $q| (p-1)$. However, this is impossible, since $q> p > 2$. If $n_{q}= p^3$, then there are $p^{3}(q-1)$ elements of order $q$.
However, this only leaves $p^{3}q-p^{3}(q-1)= p^3$ elements, and the Sylow p-subgroup must be normal, a case we already considered. Therefore,
the only remaining possibility is that $G\cong \mathbb Z_q \rtimes P$. By Propositions~\ref{genus 2} and \ref{genus 5}, we have
$P\cong \mathbb Z_{p^3}$ or $\mathbb Z_4\times \mathbb Z_2$ or $Q_8$. If $P\cong \mathbb Z_{p^3}$, then
$G\cong \mathbb Z_q\rtimes \mathbb Z_{p^3}=\langle a,b~|~a^q=b^{p^3}=1, bab^{-1}=a^i, \mbox{ord}_q(i)=p^t\rangle$, $p^t~|~(q-1)$ and
$H_1:=\langle a\rangle$,
$H_2:=\langle a,b^p\rangle$, $H_3:=\langle ab^{p^2}\rangle$, $H_4:=\langle b\rangle$, $H_5:=\langle b^p\rangle$, $H_6:=\langle b^{p^2}\rangle$,
$H_7:=\langle ab\rangle$,
$H_8:=\langle a^2b\rangle$, $H_9:=\langle a^3b\rangle$, $H_{10}:=\langle a^4b\rangle$ are proper subgroups of $G$. Here $H_1$, $H_2$ are normal in $G$, and
$H_6$ is a subgroup of $H_2$, $H_3$, $H_4$, $H_5$, $H_7$, $H_8$, $H_9$, $H_{10}$. It follows that $K_{3,7}$ is a subgraph of
$\Gamma(G)$ with bipartition $X:=\{H_1$, $H_2$, $H_6\}$ and $Y:=\{H_3$, $H_4$, $H_5$, $H_7$, $H_8$, $H_9$, $H_{10}\}$. Therefore, $\gamma(\Gamma(G))>1$, $\overline{\gamma}(\Gamma(G))>1$. If
$P\cong \mathbb Z_4\times \mathbb Z_2$,
then, by~\eqref{e1}, $P$ together with its proper subgroups forms $K_7$ as a subgraph of $\Gamma(G)$.
Since $H:=\mathbb Z_q$
is normal in $G$, $P$ together with its proper subgroups and $H_1$ form $K_8$ as a subgraph of $\Gamma(G)$. Therefore, $\gamma(\Gamma(G))>1$, $\overline{\gamma}(\Gamma(G))>1$.
If $P\cong Q_8$, then
$H_1:=\langle a\rangle$, $H_2:=\langle b,c\rangle$, $H_3:=\langle a,b\rangle$, $H_4:=\langle a,c\rangle$,
$H_5:=\langle b\rangle$, $H_6:=\langle c\rangle$, $H_7:=\langle bc\rangle$, $H_8:=\langle b^2\rangle$ are subgroups of $G$, where $a\in \mathbb Z_q$,
$\langle b,c\rangle=Q_8$. Here $H_2$, $H_5$, $H_6$, $H_7$,
$H_8$ permute with each other; $H_1$, $H_3$, $H_4$ are normal in $G$. It follows that these eight subgroups form $K_8$ as a subgraph of $\Gamma(G)$
and so $\gamma(\Gamma(G))>1$, $\overline{\gamma}(\Gamma(G))>1$.

If $(p,q)=(2,3)$, then $G\cong S_4$, and $D_8$  is a subgroup of $S_4$. By Proposition~\ref{genus 5}, $\gamma(\Gamma(D_8))>1$, $\overline{\gamma}(\Gamma(D_8))>1$.
Thus, we have
\begin{equation}\label{e5}
\gamma(\Gamma(S_4))>1,\ \ \ \overline{\gamma}(\Gamma(S_4))>1.
\end{equation}
So, the result is true when $\alpha=3$.

Assume now $\alpha >3$ and the result is true for all the non-abelian groups of order $p^mq$, $m<\alpha$. We prove the result for $\alpha$. If $n_p(G)=1$,
then our group is isomorphic to $P\rtimes \mathbb Z_q$ with $\gamma(\Gamma(P))\leq1$. By Proposition~\ref{genus 2}, $P\cong \mathbb Z_{p^{\alpha}}$.
Then $G\cong \mathbb Z_{p^\alpha}\rtimes \mathbb Z_q=\langle a,b~|~a^{p^\alpha}=b^q=1, bab^{-1}=a^i, i^q\equiv 1(\mbox{mod}~ p^\alpha)\rangle$. It
has a subgroup $H:=\langle a^p,b\rangle\cong \mathbb Z_{p^{\alpha-1}}\rtimes \mathbb Z_q$, and, by induction hypothesis, $\gamma(\Gamma(H))>1$, $\overline{\gamma}(\Gamma(H))>1$.
Therefore, $\gamma(\Gamma(G))>1$, $\overline{\gamma}(\Gamma(G))>1$.
If $n_p\neq 1$, since $G$ is solvable, $G$ has a normal subgroup $N$ of order $p^{\alpha-1}q$ and at least one subgroup of order $p^\alpha$, say $H_1$. If $\gamma(\Gamma(N))>1$ and $\overline{\gamma}(\Gamma(N))>1$, then $\gamma(\Gamma(G))>1$ and $\overline{\gamma}(\Gamma(G))>1$. So, by Propositions~\ref{genus 2} and \ref{genus 5},
we have $N\cong \mathbb Z_{p^3q}$. Let $H_2$, $H_3$, $H_4$, $H_5$, $H_6$, $H_7$ be the subgroups of $N$ of order $p$, $p^2$, $p^3$, $q$, $pq$,
$p^2q$, respectively. Here $H_1H_5=H_1H_6=H_1H_7=G$; $N$ together with its subgroups form $K_7$ as a subgraph of $\Gamma(G)$.
It follows that these eight subgroups together form a subgraph in $\Gamma(G)$, which is isomorphic to $\mathcal A_1$ of Figure~\ref{fig:8vertex},
so $\gamma(\Gamma(G))>1$, $\overline{\gamma}(\Gamma(G))>1$.
\end{proof}

\begin{pro}\label{genus 9}
 If $G$ is a non-abelian group of order $p^2q^2$, where $p$ and $q$ are distinct primes, then $\Gamma(G)$ is toroidal and projective-planar if and only if $G\cong (\mathbb Z_3\times \mathbb Z_3)\rtimes \mathbb Z_4=\langle a,b,c~|~a^3=b^3=c^4=1, ab=ba, cac^{-1}=b^{-1}$,
$cbc^{-1}=ab^l\rangle$, where $\bigl(\begin{smallmatrix}
  0 & -1\\ 1 & l
\end{smallmatrix} \bigr)$ has order dividing $4$ in $GL_2(3)$.
\end{pro}
\begin{proof}
Here we use the classification of group of order $p^2q^2$ given in~\cite{lin}.

Let $P$ and $Q$ be a Sylow $p$-subgroup and Sylow $q$-subgroup of $G$, respectively. Without loss of generality, we assume that $p>q$.
By Sylow's Theorem, $n_p=1,q,q^2$. However, $n_p=q$ is not possible since $p>q$. If $n_p=q^2$, then $p|(q+1)(q-1)$. This implies that $p|(q+1)$, which is possible only when $(p,q)=(3,2)$.

When $(p,q)\neq (3,2)$, then $G\cong P\rtimes Q$.

If $G\cong \mathbb Z_{p^2}\rtimes\mathbb Z_{q^2}=
\langle a,b|a^{p^2}=b^{q^2}=1,bab^{-1}=a^i,i^{q^2}\equiv 1~(\text{mod}~p^2)\rangle$, then $H_1:=\langle a\rangle$,
$H_2:=\langle a^p\rangle$, $H_3:=\langle a,b^q\rangle$, $H_4:=\langle b\rangle$, $H_5:=\langle b^q\rangle$, $H_6:=\langle a^p,b^q\rangle$,
$H_7:=\langle a^p,b\rangle$, $H_8:=\langle a^pb\rangle$, $H_9:=\langle a^{2p}b\rangle$, $H_{10}:=\langle a^{3p}b\rangle$ are proper subgroups of $G$.
Here $H_1$, $H_2$, $H_3$ are normal in $G$.
It follows that $K_{3,7}$ is a subgraph of $\Gamma(G)$ with bipartition $X:=\{H_1$, $H_2$, $H_3\}$
and $Y:=\{H_4$, $H_5$, $H_6$, $H_7$, $H_8$, $H_9$, $H_{10}\}$. Therefore, $\gamma(\Gamma(G))>1$, $\overline{\gamma}(\Gamma(G))>1$.

If $G \cong\mathbb Z_{p^2}\rtimes(\mathbb Z_q\times\mathbb Z_q)$, then $H_1:=\langle a\rangle$, $H_2:=\langle a^p\rangle$,
$H_3:=\langle a,b\rangle$, $H_4:=\langle a,c\rangle$, $H_5:=\langle a^p,c\rangle$, $H_6:=\langle a^p,c\rangle$, $H_7:=\langle b\rangle$,
$H_8:=\langle c\rangle$, $H_9:=\langle b,c\rangle$, $H_{10}:=\langle a^p,b,c\rangle$ are proper subgroups of $G$, where
$\langle a\rangle=\mathbb Z_{p^2}$ and $\langle b,c\rangle=\mathbb Z_q\times \mathbb Z_q$. Here $H_1$, $H_3$, $H_4$ are normal
in $G$. It follows that $K_{3,7}$ is a subgraph of $\Gamma(G)$ with bipartition $X:=\{H_1$, $H_3$, $H_4\}$ and $Y:=\{H_2$, $H_5$, $H_6$, $H_7$, $H_8$, $H_9$,
$H_{10}\}$.
Therefore, $\gamma(\Gamma(G))>1$, $\overline{\gamma}(\Gamma(G))>1$.

If $G\cong(\mathbb Z_p\times\mathbb Z_p)\rtimes\mathbb Z_{q^2}$,
then $\mathbb Z_p\times\mathbb Z_p$ is a subgroup of $G$. If $p\geq 7$, then, by Proposition~\ref{genus 2}, $\gamma(\Gamma(\mathbb Z_p\times \mathbb Z_p))>1$ and $\overline{\gamma}(\Gamma(\mathbb Z_p\times \mathbb Z_p))>1$.
Therefore,
$\gamma(\Gamma(G))>1$, $\overline{\gamma}(\Gamma(G))>1$. If $p=5$, then $H:=\langle a,b\rangle=\mathbb Z_p\times \mathbb Z_p$, $H_1:=\langle a,b,c^q\rangle$ are proper normal subgroup of $G$,
where $\langle c\rangle=\mathbb Z_{q^2}$.
So, $H_1$, $H$, and the subgroups of $H$ form $K_8$ as a subgraph of $\Gamma(G)$. Therefore,
$\gamma(\Gamma(G))>1$, $\overline{\gamma}(\Gamma(G))>1$.

If $G\cong (\mathbb Z_p\times \mathbb Z_p)\rtimes (\mathbb Z_q\times \mathbb Z_q)$, then we can use the same argument as above by taking
$H:=\mathbb Z_p\times\mathbb Z_p$, $H_1:=\langle a,b,d\rangle$, where $\langle d\rangle=\mathbb Z_q$, so $\gamma(\Gamma(G))>1$, $\overline{\gamma}(\Gamma(G))>1$.

Now consider the case $(p$, $q)=(3$, $2)$. Up to isomorphism, there are nine groups of order 36. We investigate the toroidality and projective-planarity of permutability graphs of subgroups
for each of these nine groups.

\noindent\textbf{Case 1:} If $G\cong D_{18}$, then $H_1:=\langle a\rangle$, $H_2:=\langle a^2\rangle$, $H_3:=\langle a^3\rangle$, $H_4:=\langle a^6\rangle$,
$H_5:=\langle a^9\rangle$, $H_6:=\langle b\rangle$, $H_7:=\langle ba\rangle$, $H_8:=\langle ba^2\rangle$, $H_9:=\langle ba^3\rangle$,
$H_{10}:=\langle ba^4\rangle$
are subgroups of $G$. Here $H_1$, $H_2$, $H_3$ are normal in $G$, so they permute with all the subgroups of $G$. It follows that $K_{3,7}$ is
a subgraph of $\Gamma(G)$ with bipartition $X:=\{H_1$, $H_2$, $H_3\}$ and $Y:=\{H_4$, $H_5$, $H_6$, $H_7$, $H_8$, $H_9$, $H_{10}\}$, so $\gamma(\Gamma(G))>1$, $\overline{\gamma}(\Gamma(G))>1$.

\noindent\textbf{Case 2:} If $G\cong S_3\times S_3$, then $H_1:=S_3\times \{e\}$, $H_2:=\{e\}\times S_3$, $H_3:=\langle (123)\rangle\times S_3$,
$H_4:=\{e\}\times\langle (123)\rangle$,
$H_5:=\{e\}\times\langle (23)\rangle$, $H_6:=\langle (23)\rangle\times \{e\}$, $H_7:=\{e\}\times\langle (13)\rangle$, $H_8:=\langle (13)\rangle\times \{e\}$,
$H_9:=\langle (12)\rangle\times \{e\}$, $H_{10}:=\{e\}\times\langle (12)\rangle$ are proper subgroups of $G$. Here $H_1$, $H_2$, $H_3$ permute with all
the subgroups of $G$. It follows that $K_{3,7}$ is a subgraph of $\Gamma(G)$. Hence $\gamma(\Gamma(G))>1$, $\overline{\gamma}(\Gamma(G))>1$.

\noindent\textbf{Case 3:} If $G\cong \mathbb Z_3\times A_4$, then $H_1:=\mathbb Z_3\times \{e\}$, $H_2:=\{e\}\times A_4$, $H_3:=\{e\}\times \langle (12)(34)$, $(13)(24)\rangle$,
$H_4:=\mathbb Z_3\times \langle (12)(34)$, $(13)(24)\rangle$, $H_5:=\mathbb Z_3\times\langle (12)(34)\rangle$, $H_6:=\{e\}\times \langle (12)(34)\rangle$, $H_7:=\{e\}\times \langle (13)(24)\rangle$,
$H_8:=\{e\}\times \langle (14)(23)\rangle$,
$H_9:=\mathbb Z_3\times \langle (14)(23)\rangle$ are
subgroups of $G$. Here $H_1$, $H_2$, $H_3$, $H_4$ permute with $H_5$, $H_6$, $H_7$, $H_8$, $H_9$. It follows that $K_{4,5}$ is a
subgraph of $\Gamma(G)$. Therefore, $\gamma(\Gamma(G))>1$, $\overline{\gamma}(\Gamma(G))>1$.

\noindent\textbf{Case 4:} If $G\cong \mathbb Z_6\times S_3$, then $H_1:=\mathbb Z_6\times \{e\}$, $H_2:=\{e\}\times S_3$, $H_3 :=\mathbb Z_3\times S_3$,
$H_4:=\mathbb Z_2\times S_3$, $H_5:=\mathbb Z_2\times \{e\}$, $H_6:=\mathbb Z_3\times \{e\}$, $H_7:=\mathbb Z_3\times\langle (12)\rangle$,
$H_8:=\mathbb Z_2\times\langle (12)\rangle$, $H_9:=\{e\}\times\langle (12)\rangle$ are proper subgroups of $G$. Here $H_1$, $H_2$, $H_3$, $H_4$
permute with
$H_5$, $H_6$, $H_7$, $H_8$, $H_9$. It follows that $K_{4,5}$ is a subgraph of $\Gamma(G)$ with bipartition $X:=\{H_1$, $H_2$, $H_3$, $H_4\}$ and
$Y:=\{H_5$, $H_6$, $H_7$, $H_8$, $H_9\}$, so $\gamma(\Gamma(G))>1$, $\overline{\gamma}(\Gamma(G))>1$.

\noindent\textbf{Case 5:} If $G\cong \mathbb Z_9\rtimes\mathbb Z_4=\langle a,b~|~a^9=b^4=1, bab^{-1}=a^i, i^4\equiv 1(\mbox{mod}~ 9)\rangle$,
then $H_1:=\langle a\rangle$, $H_2:=\langle a^3\rangle$, $H_3:=\langle b\rangle$,
$H_4:=\langle b^2\rangle$, $H_5:=\langle a,b^2\rangle$, $H_6:=\langle a^3,b\rangle$, $H_7:=\langle a^3,b^2\rangle$, $H_8:=\langle ab,b^2\rangle$,
$H_9:=\langle a^2b,b^2\rangle$, $H_{10}:=\langle a^3b,b^2\rangle$ are proper subgroups of $G$.
Since $H_1$, $H_2$, $H_5$ are normal in $G$, and $H_4$ is a subgroup of $H_i$, $i \neq 1,2$, $K_{3,7}$ is a subgraph of $\Gamma(G)$ with bipartition
$X=\{H_1, H_2, H_5\}$ and $Y=\{H_3, H_4, H_6, H_7, H_8, H_9, H_{10}\}$. It follows that $\gamma(\Gamma(G))>1$, $\overline{\gamma}(\Gamma(G))>1$.


\noindent\textbf{Case 6:} If $G\cong \mathbb Z_3\times(\mathbb Z_3\rtimes \mathbb Z_4)=\langle a,b,c~|~a^3=b^3=c^4=1, ab=ba, ac=ca, cbc^{-1}=b^i,
\mbox{ord}_2(i)=3\rangle$, then $H_1:=\langle a\rangle$, $H_2:=\langle b\rangle$,
$H_3:=\langle c\rangle$, $H_4:=\langle c^2\rangle$, $H_5:=\langle bc\rangle$, $H_6:=\langle b^2c\rangle$, $H_7:=\langle a\rangle\times \langle b^2c\rangle$,
$H_8:=\langle a\rangle\times \langle b\rangle$, $H_9:=\langle a\rangle\times \langle c^2\rangle$ are proper subgroups of $G$. Here
$H_1$, $H_7$, $H_8$, $H_9$ permute with $H_2$, $H_3$, $H_4$, $H_5$, $H_6$. It follows that $K_{4,5}$ is a subgraph of $\Gamma(G)$ with bipartition
$X:=\{H_1$, $H_7$, $H_8$, $H_9\}$ and $Y:=\{H_2$, $H_3$, $H_4$, $H_5$, $H_6\}$. Therefore, $\gamma(\Gamma(G))>1$, $\overline{\gamma}(\Gamma(G))>1$.

\noindent\textbf{Case 7:} If $G\cong (\mathbb Z_3\times \mathbb Z_3)\rtimes \mathbb Z_4=\langle a,b,c~|~a^3=b^3=c^4=1, ab=ba, cac^{-1}=a^ib^j$,
$cbc^{-1}=a^kb^l\rangle$, where $\bigl(\begin{smallmatrix}
  i & j\\ k & l
\end{smallmatrix} \bigr)$ has order dividing $4$ in $GL_2(3)$, then we need to consider the following subcases.

\noindent\textbf{Subcase 7a:} Suppose $G$ has subgroups of order $pq^2$ and $pq$, where $p=3, q=2$. Then $H_1:=\langle a,b\rangle$, $H_2:=\langle a,b,c^2\rangle$, $H_3:=\langle a,c\rangle$, $H_4:=\langle a,c^2\rangle$, $H_5:=\langle b,c\rangle$, $H_6:=\langle b,c^2\rangle$, $H_7:=\langle c\rangle$,
$H_8:=\langle c^2\rangle$
, $H_9:=\langle a\rangle$, $H_{10}:=\langle b\rangle$ are proper subgroups of $G$. Here $H_1$ and $H_2$ are normal in $G$;
$H_8$ is a subgroup of $H_3$, $H_4$, $H_5$, $H_6$, $H_7$, $H_9$, $H_{10}$. It follows that $K_{3,7}$ is a subgraph of
$\Gamma(G)$, so $\gamma(\Gamma(G))>1$, $\overline{\gamma}(\Gamma(G))>1$.

\noindent\textbf{Subcase 7b:} If $G$ has no subgroups of order $pq^2$ or $pq$, where $p=3, q=2$, then
$G\cong (\mathbb Z_3\times \mathbb Z_3)\rtimes \mathbb Z_4=\langle a,b,c~|~a^3=b^3=c^4=1, ab=ba, cac^{-1}=b^{-1}$,
$cbc^{-1}=ab^l\rangle$, where $\bigl(\begin{smallmatrix}
  0 & -1\\ 1 & l
\end{smallmatrix} \bigr)$ has order dividing $4$ in $GL_2(3)$. Here $G$ has unique subgroups of order $p^2$ and $p^2q$, say $H$ and $N$, respectively, $p^2$ subgroups of order $q^2$, denoted by $A_i$, $i=1,\ldots,9$, and $p^2$ subgroups of order $q$, denoted by $B_i$, $i=1,\ldots,9$. Moreover, $H$ has four subgroups of order $p$, denoted by $H_1$, $H_2$, $H_3$, $H_4$. These are the only subgroups of $G$. It follows that
\begin{equation}\label{e1111}
\Gamma(G)\cong K_2+(K_4\cup 9K_2),
 \end{equation}
 which is both toroidal and projective-planar. The corresponding toroidal and projective-planar embeddings are shown in Figures~\ref{fig:8} and~\ref{fig:9}, respectively.

\begin{figure}[ ht ]
 \begin{center}
 \includegraphics[scale=1]{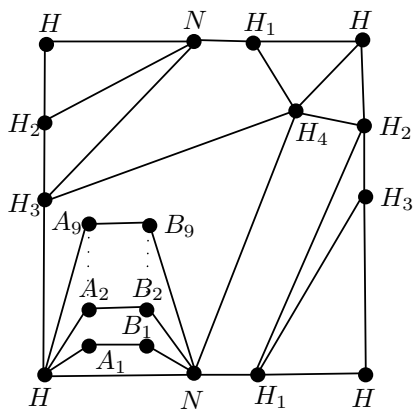}
 \caption{A torodial embedding of $\Gamma((\mathbb Z_3\times \mathbb Z_3))\rtimes \mathbb Z_4)$.\label{fig:8}}
 \end{center}
 \end{figure}

 \begin{figure}[ ht ]
 \begin{center}
 \includegraphics[scale=1]{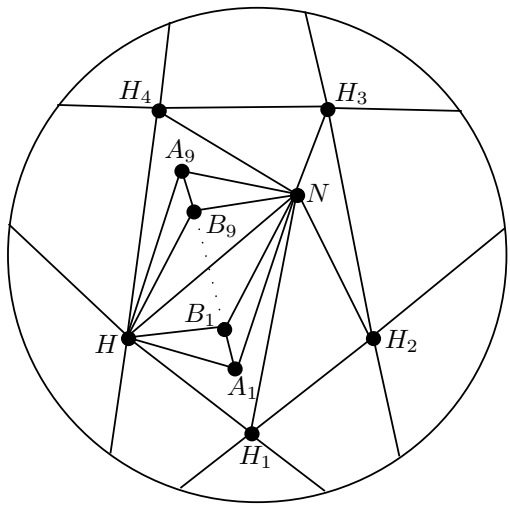}
 \caption{A projective-planar embedding of $\Gamma((\mathbb Z_3\times \mathbb Z_3)\rtimes \mathbb Z_4)$.\label{fig:9}}
 \end{center}
 \end{figure}


\noindent\textbf{Case 8:} If $G\cong \mathbb Z_2\times((\mathbb Z_3\times \mathbb Z_3)\rtimes \mathbb Z_2)=\langle a,b,c,d~|~a^2=b^3=c^3=d^2=1$,
$ab=ba, ac=ca, ad=da, bc=cb, dbd^{-1}=b^ic^j, dcd^{-1}=b^kc^l\rangle$, where $\bigl(\begin{smallmatrix}
  i & j\\ k & l
\end{smallmatrix} \bigr)$ has order $2$ in $GL_2(3)$, then $H_1:=\langle a,b,c\rangle$,
 $H_2:=\langle b,c,d\rangle$, $H_3:=\langle b\rangle$, $H_4:=\langle a,b\rangle$, $H_5:=\langle a,c\rangle$, $H_6:=\langle a,b,d\rangle$,
$H_7:=\langle b,c\rangle$, $H_8:=\langle b,d\rangle$, $H_9:=\langle c,d\rangle$, $H_{10}:=\langle a\rangle$ are proper subgroups of $G$. Here
$H_1$ and $H_2$ are normal in $G$; $H_3$ is a subgroup of $H_4$, $H_6$, $H_7$, $H_8$; $H_3H_5=H_1$;
$H_3H_9=H_2$; $H_3H_{10}=H_4$. It follows that $K_{3,7}$ is a subgraph of $\Gamma(G)$ with bipartition $X:=\{H_1$, $H_2$, $H_3\}$ and
$Y:=\{H_4$, $H_5$, $H_6$, $H_7$, $H_8$, $H_9$, $H_{10}\}$. Therefore, $\gamma(\Gamma(G))>1$, $\overline{\gamma}(\Gamma(G))>1$.

\noindent\textbf{Case 9:} If $G\cong (\mathbb Z_2\times \mathbb Z_2)\rtimes \mathbb Z_9$, then $H_1:=\langle a,c\rangle$, $H_2:=\langle b,c\rangle$,
$H_3:=\langle a,b\rangle$, $H_4:=\langle a\rangle$, $H_5:=\langle b\rangle$, $H_6:=\langle c\rangle$, $H_7:=\langle c^2\rangle$,
$H_8:=\langle a,c^2\rangle$, $H_9:=\langle b,c^2\rangle$, $H_{10}:=\langle a,b,c^2\rangle$ are proper subgroups of $G$, where
$\langle a,b\rangle=\mathbb Z_2\times \mathbb Z_2$, $\langle c\rangle=\mathbb Z_9$. Here $H_1$, $H_2$, $H_3$ are
normal in $G$, so we have $K_{3,7}$ as a subgraph in $\Gamma(G)$ with bipartition
$X:=\{H_1$, $H_2$, $H_3\}$ and $Y:=\{H_4$, $H_5$, $H_6$, $H_7$, $H_8$, $H_9$, $H_{10}\}$. Therefore, $\gamma(\Gamma(G))>1$, $\overline{\gamma}(\Gamma(G))>1$.

Combining all the cases together, the result follows.
\end{proof}

\begin{pro}\label{genus 10}
If $G$ is a non-abelian group of order $p^\alpha q^\beta$, where $p,q$ are distinct primes, $\alpha$, $\beta\geq 2$, and $\alpha+\beta\geq5$, then
$\gamma(\Gamma(G))>1$, $\overline{\gamma}(\Gamma(G))>1$.
\end{pro}
\begin{proof}
We prove this result by induction on $\alpha+\beta$. If $\alpha+\beta=5$, then $|G|=p^2q^3$. Since $G$ is solvable,
it has a normal subgroup $N$ of prime index.

\noindent\textbf{Case 1:} If $[G:N]=q$, then $|N|=p^2q^2$. If $\gamma(\Gamma(N))>1$, $\overline{\gamma}(\Gamma(N))>1$, then $\gamma(\Gamma(G))>1$, $\overline{\gamma}(\Gamma(G))>1$.
By Propositions~\ref{genus 2} and \ref{genus 9}, here $N\cong \mathbb Z_{p^2q^2}$ or $(\mathbb Z_3\times \mathbb Z_3)\rtimes \mathbb Z_4$.
If $N\cong \mathbb Z_{p^2q^2}$, then $N$ together with its proper subgroups form $K_8$ as a subgraph of $\Gamma(G)$. If $N\cong (\mathbb Z_3\times \mathbb Z_3)\rtimes \mathbb Z_4$, then, by \eqref{e1111}, the subgraph generated by $N$ and its proper subgroups in $\Gamma(G)$ contains $K_{3,7}$ as a subgraph.
It follows that $\gamma(\Gamma(G))>1$, $\overline{\gamma}(\Gamma(G))>1$.

\noindent\textbf{Case 2:} If $[G:N]=p$, then $|N|=pq^3$. If $\gamma(\Gamma(N))>1$, $\overline{\gamma}(\Gamma(N))>1$, then $\gamma(\Gamma(G))>1$, $\overline{\gamma}(\Gamma(G))>1$. By Propositions~\ref{genus 2} and
\ref{genus 8}, here $N\cong \mathbb Z_{pq^3}$. Let $H_1$, $H_2$, $H_3$, $H_4$, $H_5$, $H_6$ be the subgroups of $N$ of order $p$, $q$, $q^2$, $q^3$, $pq$, $pq^2$, respectively. Let $P$ be a Sylow $p$-subgroup of $G$ containing $H_1$. Consider the subgroup $H:=\langle P, H_2\rangle$ of $G$. Here $N$ together with its proper subgroups forms $K_7$ as a subgraph of $\Gamma(G)$. Also,
$HH_4=G$, $H_1$, $H_2$ are subgroups of $H$. It follows that these subgroups form a subgraph in $\Gamma(G)$, which is isomorphic to $\mathcal A_1$ of Figure~\ref{fig:8vertex},
so $\gamma(\Gamma(G))>1$, $\overline{\gamma}(\Gamma(G))>1$.

Now assume that $\alpha +\beta > 5$ and the result is
true for all the non-abelian group  of order $p^mq^n$, where $m+n<\alpha+\beta$ ($m+n\geq 5$, $m,n\geq 2$). We prove this result for $\alpha+\beta$. Since $G$ is solvable, then $G$ has  a normal
subgroup $H$ with a prime index, say $q$, and so $|H|=p^\alpha q^{\beta-1}$. If $H$ is abelian, then by Proposition~\ref{genus 2},
$\gamma(\Gamma(H))>1$, $\overline{\gamma}(\Gamma(H))>1$. If $H$ is non-abelian, then we have the following cases:

\noindent \textbf{Case a:} If $\beta=2$, then $\alpha>2$, and by Proposition~\ref{genus 8}, $\gamma(\Gamma(H))>1$, $\overline{\gamma}(\Gamma(H))>1$.

\noindent \textbf{Case b:} If $\beta>2$, then, by the induction hypothesis, $\gamma(\Gamma(H))>1$, $\overline{\gamma}(\Gamma(H))>1$.

\noindent \textbf{Case c:} If $\alpha=2$, then $\beta >3$, and by Case b, $\gamma(\Gamma(H))>1$, $\overline{\gamma}(\Gamma(H))>1$.

\noindent \textbf{Case d:} If $\alpha>2$, then, by the induction hypothesis, $\gamma(\Gamma(H))>1$, $\overline{\gamma}(\Gamma(H))>1$.

It follows that $\gamma(\Gamma(G))>1$, $\overline{\gamma}(\Gamma(G))>1$.
Combining all the cases together completes the proof.
\end{proof}


\begin{pro}\label{genus 11}
If $G$ is a non-abelian solvable group of order $p^\alpha q^\beta r^\gamma$, where $p,q,r$ are distinct primes, then $\gamma(\Gamma(G))>1$, $\overline{\gamma}(\Gamma(G))>1$.
\end{pro}
\begin{proof}
Since $G$ is solvable, it has a Sylow basis $\{P$, $Q$, $R\}$, where $P$, $Q$, $R$ are Sylow $p$, $q$, $r$-subgroups of $G$, respectively. We split the proof into several cases.\\
\noindent\textbf{Case 1:} If $\alpha=\beta=\gamma=1$, then consider the following subcases. Without loss of generality, we assume that $p<q<r$. Here the Sylow $r$-subgroup of $G$ is always unique, i.e. $n_r=1$.\\
\noindent\textbf{Subcase 1a:} Suppose $n_p=n_q=1$. Then
$G$ is abelian, which is not possible.

\noindent\textbf{Subcase 1b:} $n_p\neq 1$ and $n_q=1$.
Let $P_1$, $P_2$, $P_3$ be Sylow $p$-subgroups of $G$. Here $Q$, $R$ are normal in $G$, and so $QR$ is also normal in $G$. By \cite[pp.\,216--219]{cole}, $G$ has $q$ subgroups  either of order $pq$ or $pr$. If $G$ has $q$ subgroups of order $pq$, then
$\Gamma(G)$ contains $K_{3,7}$ as a subgraph with bipartition $X:=\{Q$, $R$, $QR\}$ and
$Y:=\{P_1$, $P_2$, $P_3$, $QP_1$, $QP_2$, $QP_3$, $RP_1\}$. If $G$ has $q$ subgroups of order $pr$, then
$\Gamma(G)$ contains $K_{3,7}$ as a subgraph with bipartition $X:=\{Q$, $R$, $QR\}$ and
$Y:=\{P_1$, $P_2$, $P_3$, $RP_1$, $RP_2$, $RP_3$, $QP_1\}$. Therefore, $\gamma(\Gamma(G))>1$, $\overline{\gamma}(\Gamma(G))>1$.

\noindent\textbf{Subcase 1c:} $n_p\neq 1$ and $n_q\neq 1$. Let $P_1$, $P_2$, $P_3$ be Sylow $p$-subgroups of $G$, and
$Q_1$, $Q_2$, $Q_3$ be Sylow $q$-subgroups of $G$. Here $R$ is normal in $G$. By \cite[pp.\,219--220]{cole}, $G$ has $q$ subgroups of order $pq$ (denote one of them by $H_1$) and unique normal subgroups of order $qr$ and $pr$, say $H_2$ and $H_3$, respectively. It follows that $\Gamma(G)$ contains $K_{3,7}$ as a subgraph with
bipartition $X:=\{R$, $H_2$, $H_3\}$ and $Y:=\{P_1$, $P_2$, $P_3$, $Q_1$, $Q_2$, $Q_3$, $H_1\}$, so $\gamma(\Gamma(G))>1$, $\overline{\gamma}(\Gamma(G))>1$.


\noindent\textbf{Case 2:} If $\alpha=2$ and $\beta=\gamma=1$, then $H_1:=PQ$, $H_2:=PR$ are two proper subgroups of $G$ of order $p^2q$ and $p^2r$, respectively.
Here $P$, $Q$, $R$, $QR$, $H_1$, $H_2$ permute with each other.
If $P \cong Z_{p^2}$, then $G$ has subgroups of order $p$ and $pq$, say $H_3$ and $H_4$, respectively. Here $H_3$ is a subgroup of $H_4$, and they permute with $P$, $H_1$, $H_2$. If $P \cong Z_p \times Z_p$, then $P$ has at least two subgroups of order $p$, say $H_3$, $H_4$. Here $H_3$, $H_4$ permute with each other and $H_1$, $H_2$.
It follows that these subgroups form a subgraph in $\Gamma(G)$ isomorphic to $\mathcal A_1$ of Figure~\ref{fig:8vertex}.
So, $\gamma(\Gamma(G))>1$, $\overline{\gamma}(\Gamma(G))>1$.

\noindent\textbf{Case 3:} If $\alpha=3$ and $\beta=\gamma=1$, then $H_1=PQ$, $H_2=PR$ are two proper subgroups of $G$ of order $p^3q$ and $p^3r$, respectively.
Here $P$, $Q$, $R$, $QR$, $H_1$, $H_2$ permute with each other; also,
$H_1$ has subgroups of order $p$, $p^2$, say $H_3$, $H_4$, respectively. However, $H_3$ is a subgroup of $H_4$; they permute with $H_1$, $H_2$, $P$. It follows that these subgroups form a subgraph in $\Gamma(G)$ isomorphic to $\mathcal A_1$ of Figure~\ref{fig:8vertex}.
So, $\gamma(\Gamma(G))>1$, $\overline{\gamma}(\Gamma(G))>1$.

\noindent\textbf{Case 4:} Suppose $\alpha+\beta+\gamma\geq 5$, and, without loss of generality, we assume that $\alpha\geq\beta\geq\gamma$. Consider the subgroup $H:=PQ$ of $G$ of
order $p^\alpha q^\beta$.
If $\gamma(\Gamma(H))>1$, $\overline{\gamma}(\Gamma(H))>1$, then $\gamma(\Gamma(G))>1$, $\overline{\gamma}(\Gamma(G))>1$. By Propositions~\ref{genus 2} and \ref{genus 9}, we have here $H\cong \mathbb Z_{p^2q^2}$ or $(\mathbb Z_3\times \mathbb Z_3)\rtimes \mathbb Z_4$. If $H\cong \mathbb Z_{p^2q^2}$,
then, by~Lemma~\ref{l1},
$H$ together with its subgroups form $K_8$ as a
subgraph, so $\gamma(\Gamma(G))>1$, $\overline{\gamma}(\Gamma(G))>1$. If $H\cong (\mathbb Z_3\times \mathbb Z_3)\rtimes \mathbb Z_4$, then, by \eqref{e1111}, the subgraph generated by $H$ and its proper subgroups in $\Gamma(G)$ contains $K_{3,7}$ as a subgraph.
It follows that, $\gamma(\Gamma(G))>1$, $\overline{\gamma}(\Gamma(G))>1$.
\end{proof}

\begin{pro}\label{genus 12}
 If $G$ is a solvable group whose order has more than three distinct prime factors, then $\gamma(\Gamma(G))>1$, $\overline{\gamma}(\Gamma(G))>1$.
\end{pro}
\begin{proof}
Since $G$ is solvable, it has a Sylow basis containing $P$, $Q$, $R$, $S$, where $P$, $Q$, $R$, $S$ are Sylow $p$, $q$, $r$, $s$-subgroups of $G$,
respectively. Then $P$, $Q$, $R$, $S$, $PQR$, $PQS$, $PRS$, $QRS$ are proper subgroups of $G$, and they permute with each other. So, they form $K_8$ as a
subgraph of $\Gamma(G)$, hence $\gamma(\Gamma(G))>1$, $\overline{\gamma}(\Gamma(G))>1$.
\end{proof}

\subsection{Non-solvable groups}\label{sec:5}

\begin{pro}\label{genus 1}(\cite[Theorem 2.1]{raj})
Let $G$ be a group and $N$ be a subgroup of $G$. If
$N$ is  normal in $G$, then $\Gamma (G/N)$ is isomorphic (as a graph) to  a subgraph of $\Gamma (G)$.
\end{pro}

\begin{cor}\label{c1}
If $\gamma(\Gamma(G/N))> 1$, $\overline{\gamma}(\Gamma(G/N))>1$, then $\gamma(\Gamma(G))> 1$, $\overline{\gamma}(\Gamma(G))>1$.
\end{cor}

A \textit{minimal simple group} is a non-abelian group in which all of its proper subgroups are solvable.
It is well known that any non-solvable group has a simple group as a sub-quotient, and every simple group has a minimal simple group as a sub-quotient.
Therefore, if we can show that the minimal simple groups have non-toroidal (non-projective-planar) permutability graphs,
then, by Corollary~\ref{c1}, the permutability graph of a non-solvable group is non-toroidal (resp., non-projective-planar).

The classification of minimal simple groups is given in the following result.

\begin{thm}\label{genus 13}(\cite[Corollary 1]{thomp})
A finite group is a minimal simple group if and only if it is isomorphic to one of the following:
\begin{enumerate}[{\normalfont (i)}]
\item $L_2(2^{p})$, where $p$ is any prime;
\item $L_2(3^{p})$, where  $p$ is an odd prime;
\item $ L_3(3)$;
\item $L_2(p)$, where $p$ is any prime exceeding $3$ such that $p^2  +1 \equiv 0~(\text{mod}~ 5)$;
\item $Sz(2^q)$, where $q$ is any odd prime.
\end{enumerate}
\end{thm}

\begin{lemma}\label{l2}
If $n > 2$, then $\gamma(\Gamma(D_{4n}))>1$, $\overline{\gamma}(\Gamma(D_{4n}))>1$.
\end{lemma}
\begin{proof}
 Here $H_1:=\langle a\rangle$, $H_2:=\langle a^2\rangle$, $H_3:=\langle a^2,b\rangle$, $H_4:=\langle a^2,ba\rangle$, $H_5:=\langle b\rangle$,
 $H_6:=\langle ba\rangle$, $H_7:=\langle ba^2\rangle$, $H_8:=\langle ba^3\rangle$, $H_9:=\langle ba^4\rangle$ are proper
subgroups of $D_{4n}$. Since $H_1$, $H_2$, $H_3$, $H_4$ are normal in $D_{4n}$, it follows
that $K_{4,5}$ is a subgraph of $\Gamma(G)$. Therefore, $\gamma(\Gamma(D_{4n}))>1$, $\overline{\gamma}(\Gamma(D_{4n}))>1$.
\end{proof}

\begin{pro}\label{genus 14}
 If $G$ is a non-solvable group, then $\gamma(\Gamma(G))>1$, $\overline{\gamma}(\Gamma(G))>1$.
\end{pro}
\begin{proof}
As mentioned above, it is enough to investigate the toroidality and projective-planarity of permutability graphs of subgroups of minimal simple groups.

\noindent \textbf{Case 1:}  $G \cong L_2(q^p)$.
If $p=2$, then the only non-solvable group is $ L_2(4)$. Also, $ L_2(4) \cong A_5$~\cite{atlas}.
Note that $ A_5$ contains four copies of $A_4$, say $H_i$, $i=1,2,3,4$, and five copies of $\mathbb Z_5$, say $H_j$, $j=5,6,7,8,9$, as its subgroups.
Here for each $i=1,2,3,4$ and $j=5,6,7,8,9$, $H_iH_j=A_5$. It follows that $K_{4,5}$ is a subgraph of
$\Gamma(G)$ with bipartition $X:=\{H_1$, $H_2$, $H_3$, $H_4\}$ and
$Y:=\{H_5$, $H_6$, $H_7$, $H_8$, $H_9\}$, and so $\gamma(\Gamma(G))>1$, $\overline{\gamma}(\Gamma(G))>1$.
If $p> 2$, then $ L_2(q^p)$ contains a subgroup isomorphic to $(\mathbb Z_q)^p$, namely the subgroup of
matrices of the form $\overline{
\left( \begin{smallmatrix}
1 & a\\
0 & 1
\end{smallmatrix}\right)
}$ with $a \in \mathbb F_{q^p}$. By Proposition~\ref{genus 2}, $\gamma(\Gamma((\mathbb Z_q)^p))>1$, $\overline{\gamma}(\Gamma((\mathbb Z_q)^p))>1$.
 Therefore, $\gamma(\Gamma(G))>1$, $\overline{\gamma}(\Gamma(G))>1$.

\noindent \textbf{Case 2:} $G \cong  L_3(3)$. In $SL_3(3)$, the only matrix in the subgroup $H$ is the identity matrix, so $L_3(3) \cong SL_3(3)$.
Let us consider the subgroup consisting of matrices of the form
$\left(\begin{smallmatrix}
1 & a & b\\
0 & 1 & c\\
0 & 0 & 1
\end{smallmatrix}\right)$
with $a$, $b$, $c \in \mathbb F_3$. This subgroup is isomorphic to the group
$(\mathbb Z_p \times \mathbb Z_p) \rtimes \mathbb Z_p$ with $p=3$. By Proposition~\ref{genus 5},
$\gamma(\Gamma ((\mathbb Z_p \times \mathbb Z_p) \rtimes \mathbb Z_p))>1$, $\overline{\gamma}(\Gamma ((\mathbb Z_p \times \mathbb Z_p) \rtimes \mathbb Z_p))>1$, and so $\gamma(\Gamma (G))>1$, $\overline{\gamma}(\Gamma(G))>1$.

\noindent \textbf{Case 3:} $G \cong L_2(p)$. We have to consider two subcases:

\noindent \textbf{Subcase 3a:} $p \equiv 1  ~(\text{mod}~ 4)$. Then  $L_2(p)$
has a  subgroup isomorphic to $D_{p-1}$ \cite[p.~222]{boh-reid}.
So, by Lemma~\ref{l2}, $\gamma(\Gamma (D_{p-1}))>1$ when $p>5$.
If $p=5$, then $L_2(5) \cong A_5 \cong L_2(4)$ \cite{atlas}. By Case 1, $\gamma(\Gamma (A_5))>1$, $\overline{\gamma}(\Gamma(A_5))>1$, so $\gamma(\Gamma (G))>1$, $\overline{\gamma}(\Gamma(G))>1$.

\noindent \textbf{Subcase 3b:} $p \equiv 3 ~(\text{mod}~ 4)$.  $L_2(p)$
has  a subgroup isomorphic to $D_{p+1}$ \cite[p.~222]{boh-reid}.  By Lemma~\ref{l2}, $\gamma(\Gamma(D_{p+1}))>1$, $\overline{\gamma}(\Gamma(D_{p+1}))>1$ when $p > 7$.
If $p=7$, then $S_4$ is a maximal subgroup of $L_2(7)$~\cite{atlas}. By~\eqref{e5}, $\gamma(\Gamma(S_4))>1$, $\overline{\gamma}(\Gamma(S_4))>1$,
and so $\gamma(\Gamma(G))>1$, $\overline{\gamma}(\Gamma(G))>1$.

\noindent \textbf{Case 4:}  $G \cong Sz(2^q)$. Then  $Sz(2^q)$ has a subgroup isomorphic to
$(\mathbb Z_2)^q$, $q \geq 3$ \cite[p.~466]{goren}. By Proposition~\ref{genus 2}, $\gamma(\Gamma((\mathbb Z_2)^q))>1$, $\overline{\gamma}(\Gamma((\mathbb Z_2)^q))>1$ for $q \geq 3$.
Therefore, $\gamma(\Gamma(G))>1$, $\overline{\gamma}(\Gamma(G))>1$.
\end{proof}


\section{Main results}\label{sec:6}

By combining all the results obtained in Sections~\ref{sec:3} and \ref{sec:4} above, we have the following general main result,
which classifies the finite groups whose permutability graphs of subgroups are toroidal or projective-planar.

\begin{thm}\label{t16}
 Let $G$ be a finite group. Then
 \begin{enumerate}[\normalfont (1)]
 \item $\Gamma(G)$ is toroidal if and only if $G$ is isomorphic to one of the following groups
(where $p,q$ and $r$ are distinct primes):
 \begin{enumerate}[{\normalfont (a)}]

 \item $\mathbb Z_{p^{\alpha}}(\alpha=6,7,8)$, $~\mathbb Z_{p^3q}$, $\mathbb Z_{p^2 q^2}$, $\mathbb Z_{pqr}$;

 \item $\mathbb Z_4\times \mathbb Z_2$, $\mathbb Z_5 \times \mathbb Z_5$;

\item  $\mathbb Z_3 \rtimes \mathbb Z_4$, $\mathbb Z_5\rtimes \mathbb Z_4$;

\item $\langle a, b, c~|~ a^p=b^p=c^q=1, ab=ba, cac^{-1}=b^{-1},
cbc^{-1}= ab^{l} \rangle$, where $\bigl(\begin{smallmatrix}
  0 & -1\\ 1 & l
\end{smallmatrix} \bigr)$ has order $q$ in $GL_2(p)$, $p=3$, 5;

 \item $(\mathbb Z_3\times \mathbb Z_3)\rtimes \mathbb Z_4=\langle a,b,c~|~a^3=b^3=c^4=1, ab=ba, cac^{-1}=b^{-1}$,
$cbc^{-1}=ab^l\rangle$, where $\bigl(\begin{smallmatrix}
  0 & -1\\ 1 & l
\end{smallmatrix} \bigr)$ has order dividing $4$ in $GL_2(3)$.
\end{enumerate}
\item $\Gamma(G)$ is projective-planar if and only if $G$ is isomorphic to one of the following groups
(where $p,q$ and $r$ are distinct primes):
 \begin{enumerate}[{\normalfont (a)}]

 \item $\mathbb Z_{p^{\alpha}}(\alpha=6,7)$, $~\mathbb Z_{p^3q}$, $\mathbb Z_{pqr}$;

 \item $\mathbb Z_4\times \mathbb Z_2$, $\mathbb Z_5 \times \mathbb Z_5$;

\item  $\mathbb Z_3 \rtimes \mathbb Z_4$;

\item $\langle a, b, c~|~ a^3=b^3=c^q=1, ab=ba, cac^{-1}=b^{-1},
cbc^{-1}= ab^{l} \rangle$, where $\bigl(\begin{smallmatrix}
  0 & -1\\ 1 & l
\end{smallmatrix} \bigr)$ has order $q$ in $GL_2(3)$;

 \item $(\mathbb Z_3\times \mathbb Z_3)\rtimes \mathbb Z_4=\langle a,b,c~|~a^3=b^3=c^4=1, ab=ba, cac^{-1}=b^{-1}$,
$cbc^{-1}=ab^l\rangle$, where $\bigl(\begin{smallmatrix}
  0 & -1\\ 1 & l
\end{smallmatrix} \bigr)$ has order dividing $4$ in $GL_2(3)$.
\end{enumerate}
\end{enumerate}
\end{thm}

The following result is a main application of the general results for group theory.

\begin{cor}\label{genus 20}
Let $G$ be a finite group, and $p,q$ are distinct primes. Then we have:
\begin{enumerate}[{\normalfont (1)}]
\item $\Gamma(G)$ is $K_{1,5}$-free if and only if $G$ is isomorphic to one of $\mathbb Z_{p^\alpha}(\alpha=2$, 3, 4, 5$)$,
$\mathbb Z_{p^\alpha q}(\alpha=1$, $2)$, $\mathbb Z_p\times \mathbb Z_p(p=2$, $3)$, $Q_8$, or $S_3$;
\item $\Gamma(G)$ is $P_5$-free if and only if
 $G$ is isomorphic to one of $\mathbb Z_{p^\alpha}(\alpha=2$, $3$, $4$, $5$, $6)$,
$\mathbb Z_{p^\alpha q}(\alpha=1$, $2)$, $\mathbb Z_p\times \mathbb Z_p(p=2$, $3)$, $Q_8$, $\mathbb Z_q\rtimes \mathbb Z_p$, or $A_4$;
\item $\Gamma(G)$ is $P_6$-free if and only if $G$ is isomorphic to one of
$\mathbb Z_{p^\alpha}(\alpha=2$, 3, 4, 5, 6, $7)$, $\mathbb Z_{p^\alpha q}(\alpha=1,2,3)$,
$\mathbb Z_{pqr}$, $\mathbb Z_p \times \mathbb Z_p(p=2$, 3, $5)$,
$\mathbb Z_4\times \mathbb Z_2$, $Q_8$, $\mathbb Z_q\rtimes \mathbb Z_p$, $A_4$, or $\langle a, b, c~|~ a^3=b^3=c^2=1, ab=ba, cac^{-1}=b^{-1},
cbc^{-1}= ab^{l} \rangle$, where $\bigl(\begin{smallmatrix}
  0 & -1\\ 1 & l
\end{smallmatrix} \bigr)$ has order $2$ in $GL_2(3)$;
\item $\Gamma(G)$ is $C_6$-free  if and only if
 $G$ is isomorphic to one of $\mathbb Z_{p^\alpha}(\alpha=2$, $3$, $4$, $5$, $6)$,
$\mathbb Z_{p^\alpha q}(\alpha=1$, $2)$, $\mathbb Z_p\times \mathbb Z_p(p=2$, $3)$, $Q_8$, $\mathbb Z_q\rtimes \mathbb Z_p$, $A_4$,
or $\langle a, b, c~|~ a^3=b^3=c^2=1, ab=ba, cac^{-1}=b^{-1},
cbc^{-1}= ab^{l} \rangle$, where $\bigl(\begin{smallmatrix}
  0 & -1\\ 1 & l
\end{smallmatrix} \bigr)$ has order $2$ in $GL_2(3)$;
\item $\Gamma(G)$ is $K_{3,3}$-free if and only if $G$ is isomorphic to one of $\mathbb Z_{p^\alpha}(\alpha=2$, $3$, $4$, $5$, $6)$,
$\mathbb Z_{p^\alpha q}(\alpha=1$, $2)$, $\mathbb Z_p\times \mathbb Z_p(p=2$, $3)$, $Q_8$, $\mathbb Z_q\rtimes \mathbb Z_p$,
$\mathbb Z_q\rtimes_2\mathbb Z_{p^2}$, $A_4$, or $\langle a, b, c~|~ a^3=b^3=c^2=1, ab=ba, cac^{-1}=b^{-1},
cbc^{-1}= ab^{l} \rangle$, where $\bigl(\begin{smallmatrix}
  0 & -1\\ 1 & l
\end{smallmatrix} \bigr)$ has order $2$ in $GL_2(3)$.

 \end{enumerate}
 \end{cor}

\begin{proof}
In the proof of Theorem~\ref{t16}(1), when we observe that $\gamma(\Gamma(G))>1$, $\Gamma(G)$ contains one of $K_{3,7}$, $K_{4,5}$, $K_8$, $\mathcal{A}_1 = K_3+ (K_3\cup K_2)$, or the graph of Figure~\ref{fig:wow} as a subgraph.
In each of these cases, $K_{3,3}$, $K_{1,5}$, $C_6$, $P_5$ and $P_6$ are subgraphs of $\Gamma(G)$. Therefore, to classify the finite groups whose permutability graph of subgroups is one of $K_{3,3}$-free, $K_{1,5}$-free, $C_6$-free, $P_5$-free or $P_6$-free,
it is enough to consider the finite groups whose permutability graph of subgroups is either planar or toroidal. Thus, we need to investigate these properties only for groups
given in Theorems~\ref{genus 19} and \ref{t16}(1).

By Theorem~\ref{genus 19} and using \eqref{e1}, \eqref{e333}, \eqref{e111}, \eqref{222}, \eqref{e112}, \eqref{e4},
the only groups $G$ such that $\Gamma(G)$ is planar
and $K_{1,5}$-free are $\mathbb Z_{p^\alpha}(\alpha=2$, 3, 4, $5)$,
$\mathbb Z_{p^\alpha q}(\alpha=1$, $2)$, $\mathbb Z_p\times \mathbb Z_p(p=2$, $3)$, $Q_8$, and $S_3$.
By Theorem~\ref{t16}(1) and using \eqref{e1}, \eqref{e8},  \eqref{e2222}, \eqref{e1111}, there is no groups $G$ such that $\Gamma(G)$
is toroidal and $K_{1,5}$-free.
Thus, the proof of $(1)$ follows.

The proofs of parts $(2)$, $(3)$, $(4)$ and $(5)$ of this Corollary are similar to the proof of part $(1)$.
Notice that the classification in parts $(3)$ and $(4)$ is an extension of the classification in part $(2)$, and the classification of part $(5)$ is an extension of part $(4)$.
\end{proof}

Next, we consider the infinite groups. It is well known that the number of subgroups of an infinite group is infinite.
Therefore, in particular, when $G$ is an infinite abelian group, $\Gamma(G)$ contains $K_8$ as a subgraph. Thus, we have the following result.

\begin{thm}\label{pbt24}
The permutability graph of subgroups of any infinite abelian group is non-toroidal and non-projective-planar.
\end{thm}


We know of no examples of infinite non-abelian groups whose permutability graph of subgroups is toroidal or projective-planar.
The question of their existence or non-existence and the study of other graph-theoretical
properties of the permutability graph of subgroups will be subjects of future work.


\end{document}